\documentclass[lettersize,journal]{IEEEtran}
\usepackage{amsmath,amsfonts}
\usepackage{algorithmic}
\usepackage{algorithm}
\usepackage{array}
\usepackage[caption=false,font=normalsize,labelfont=sf,textfont=sf]{subfig}
\usepackage{textcomp}
\usepackage{stfloats}
\usepackage{url}
\usepackage{verbatim}
\usepackage{graphicx}
\usepackage{cite}
\usepackage{eurosym}
\usepackage{mathtools}
\usepackage{amsthm}
\usepackage{multirow}
\DeclareMathOperator*{\argmin}{argmin}

\newtheorem{proposition}{Proposition}

\begin{document}

\title{A Scenario-Spatial Decomposition Approach With a Performance Guarantee for the Combined Bidding of Cascaded Hydropower and Renewables}

\author{
Luca Santosuosso,~\IEEEmembership{Graduate Student Member,~IEEE,}
Simon Camal,~\IEEEmembership{Member,~IEEE,}
Arthur Lett,
Guillaume Bontron,
Jalal Kazempour,~\IEEEmembership{Senoir Member,~IEEE,}
and Georges Kariniotakis,~\IEEEmembership{Senoir Member,~IEEE}
\thanks{L. Santosuosso, S. Camal and G. Kariniotakis are with the PERSEE Center of MINES Paris PSL University, 06904 Sophia-Antipolis, France (e-mail: luca.santosuosso@minesparis.psl.eu; simon.camal@minesparis.psl.eu; georges.kariniotakis@minesparis.psl.eu).

J. Kazempour is with the Technical University of Denmark,
2800 Kgs. Lyngby, Denmark (e-mail: jalal@dtu.dk).

G. Bontron and A. Lett are with the Compagnie Nationale du Rhône, 69004 Lyon, France (e-mail: g.bontron@cnr.tm.fr; a.lett@cnr.tm.fr).}
\thanks{}}

\markboth{}
{Shell \MakeLowercase{\textit{et al.}}: A Sample Article Using IEEEtran.cls for IEEE Journals}

\maketitle

\begin{abstract}
This study develops a scalable co-optimization strategy for the joint bidding of cascaded hydropower, wind, and solar energy units, treated as a unified entity in the day-ahead market. Although hydropower flexibility can manage the stochasticity of renewable energy, the underlying bidding problem is complex due to intricate coupling constraints and nonlinear dynamics. A decomposition in both scenario and spatial dimensions is proposed, enabling the use of distributed optimization. The proposed distributed algorithm is eventually a heuristic due to non-convexities arising from the system's physical dynamics. To ensure a performance guarantee, trustworthy upper and lower bounds on the global optimum are derived, and a mathematical proof is provided to demonstrate their existence and validity. This approach reduces the average runtime by up to 35\% compared to alternative distributed methods and by 57\% compared to the centralized optimization. Moreover, it consistently delivers solutions, whereas both centralized and alternative distributed approaches fail as the size of the optimization problem grows.
\end{abstract}


\section*{Nomenclature}
\noindent \textit{Indices and Sets}
\vspace{0.1cm}\\
\begin{tabular}{@{}l@{\hspace{0.2cm}}l@{}}
$\boldsymbol{N}$ & Set of hydropower plants, indexed by $n$\\
$\boldsymbol{S}^{\boldsymbol{\mathrm{oc}}}_n$ & \parbox[t]{8cm}{Set of inflow segments for the piecewise linear approximation of asset $n$'s operational curve, indexed by $i$}\\
$\boldsymbol{T}$ & Set of time steps, indexed by $t$\\
$\boldsymbol{\Omega}$ & Set of scenarios, indexed by $\omega$\\
$k$ & Index for iteration of the distributed algorithm
\end{tabular}

\vspace{0.1cm}
\noindent \textit{Physical Constants}
\vspace{0.1cm}\\
\begin{tabular}{@{}l@{\hspace{0.2cm}}l@{}}
$g$ & Acceleration due to gravity (m/s$^2$)\\
$w$ & Water density (kg/m$^3$)\\
\end{tabular}

\vspace{0.1cm}
\noindent \textit{Modeling Parameters}
\vspace{0.1cm}\\
\begin{tabular}{@{}l@{\hspace{0.2cm}}l@{}}
$\overline{H}_n$ & Maximum hydraulic head of asset $n$ (m)\\
$\underline{H}_n$ & Minimum hydraulic head of asset $n$ (m)\\
$P_{\omega}$ & Probability associated with scenario $\omega$ (-)\\
$\overline{P}^{\mathrm{H}}_{n}$ & Capacity of asset $n$ (MW)\\ 
$\hat{P}^{\mathrm{S}}_{\omega,t}$ & Solar farm output in scenario $\omega$ at time $t$ (MW)\\
$\hat{P}^{\mathrm{W}}_{\omega,t}$ & Wind farm output in scenario $\omega$ at time $t$ (MW)
\end{tabular}

\noindent \begin{tabular}{@{}l@{\hspace{0.2cm}}l@{}}
$\underline{Q}^{\mathrm{br}}_n$ & \parbox[t]{7.4cm}{Minimum barrage discharge of asset $n$ (m$^3$/s)}\\
$\hat{Q}^{\mathrm{ext}}_{n,\omega,t}$ & \parbox[t]{7.5cm}{Forecast of external inflow (from river and its tributaries) to asset $n$ in scenario $\omega$ at time $t$ (m$^3$/s)}\\
$Q^{\mathrm{in}}_{n,i}$ & \parbox[t]{7.4cm}{$i$-th inflow segment of asset $n$ (m$^3$/s)}\\
$\overline{Q}^{\mathrm{tr}}_n$ & \parbox[t]{7.4cm}{Maximum turbines discharge of asset $n$ (m$^3$/s)}\\
$\underline{Q}^{\mathrm{tr}}_n$ & \parbox[t]{7.4cm}{Minimum turbines discharge of asset $n$ (m$^3$/s)}\\
$S_{n,\omega,t}$ & \parbox[t]{7.4cm}{$n$-th asset surface area in scenario $\omega$ at time $t$ (m$^2$)}\\
$Z^{\mathrm{fbl},0}_{n}$ & \parbox[t]{7.4cm}{Initial forebay water level of asset $n$ (m)}\\
$\overline{Z}^{\mathrm{fbl}}_{n,i}$ & \parbox[t]{7.4cm}{$i$-th maximum forebay segment of asset $n$ (m)}\\
$\underline{Z}^{\mathrm{fbl}}_{n,i}$ & \parbox[t]{7.4cm}{$i$-th minimum forebay segment of asset $n$ (m)}\\
$Z^{\mathrm{tlr}}_n$ & \parbox[t]{7.4cm}{Tailrace water level of asset $n$ (m)}\\
$\Delta_T$ & Sampling period (s)\\
$\Delta q^{\mathrm{tr}}_n$ & \parbox[t]{7.4cm}{Turbines outflow ramp limit of asset $n$ (m$^3$/s)}\\
$\eta^{\mathrm{tr}}_{n}$ & \parbox[t]{7.4cm}{Turbine efficiency of asset $n$ (-)}\\
$\hat{\pi}^{\mathrm{E}}_{\omega, t}$ & \parbox[t]{7.5cm}{Forecast of day-ahead energy price in scenario $\omega$ at time $t$ (\euro/MWh)}\\
$\hat{\pi}^{\mathrm{E}\uparrow}_{\omega, t}$ & \parbox[t]{7.5cm}{Forecast of positive imbalance penalty price in scenario $\omega$ at time $t$ (\euro/MWh)}\\
$\hat{\pi}^{\mathrm{E}\downarrow}_{\omega, t}$ & \parbox[t]{7.5cm}{Forecast of negative imbalance penalty price in scenario $\omega$ at time $t$ (\euro/MWh)}\\
$\tau^{\mathrm{br}}_{a,b}$ & \parbox[t]{7.5cm}{Barrage water travel time from asset $a$ to $b$ (s)}\\
$\tau^{\mathrm{tr}}_{a,b}$ & \parbox[t]{7.5cm}{Turbined water travel time from asset $a$ to $b$ (s)}\\
\end{tabular}

\vspace{0.1cm}
\noindent \textit{Notation for the Proposed Distributed Algorithm}
\vspace{0.1cm}\\
\noindent \begin{tabular}{@{}l@{\hspace{0.2cm}}l@{}}
$\overline{K}$ & \parbox[t]{7.5cm}{Maximum number of iterations allowed (-)}\\
$J$ & \parbox[t]{7.5cm}{Objective function of the bidding problem (\euro)}\\
$J^{\mathrm{LB}}$ & \parbox[t]{7.5cm}{Lower bound on the problem's global optimum (\euro)}\\
$J^{\mathrm{UB}}$ & \parbox[t]{7.5cm}{Upper bound on the problem's global optimum (\euro)}\\
$J^{\mathrm{gap}}$ & \parbox[t]{7.5cm}{Gap between upper and lower bounds (\%)}\\
$\epsilon^{\mathrm{gap}}$ & \parbox[t]{7.5cm}{Tolerance parameter of termination criterion (\%)}\\
$\epsilon^{\mathrm{UB}}$ & \parbox[t]{7.5cm}{Tolerance parameter of upper bound update (\%)}\\
$\rho$ & \parbox[t]{7.5cm}{Penalty parameter (-)}\\
\end{tabular}

\vspace{0.1cm}
\noindent \textit{Variables}
\vspace{0.1cm}\\
\begin{tabular}{@{}l@{\hspace{0.2cm}}l@{}}
$b^{\mathrm{br}}_{n,\omega,t}$ & \parbox[t]{7.6cm}{Binary variable associated with the barrage safety constraint of asset $n$ in scenario $\omega$ at time $t$ (-)}\\
$b^{\mathrm{oc}}_{n,\omega,t,i}$ & \parbox[t]{7.6cm}{Binary variable of the $i$-th operational curve segment of asset $n$ in scenario $\omega$ at time $t$ (-)}\\
$e_{\omega, t}$ & \parbox[t]{7.6cm}{Energy offer in scenario $\omega$ at time $t$ (MWh)}\\
$h_{n, \omega, t}$ & \parbox[t]{7.6cm}{Hydraulic head of asset $n$ in scenario $\omega$ at time $t$ (m)}\\
$p^{\mathrm{H}}_{n, \omega, t}$ & \parbox[t]{7.6cm}{Output of asset $n$ in scenario $\omega$ at time $t$ (MW)}\\
$q^{\mathrm{br}}_{n, \omega, t}$ & \parbox[t]{7.6cm}{Outflow of barrage $n$ in scenario $\omega$ at time $t$ (m$^3$/s)}\\
$q^{\mathrm{in}}_{n,\omega,t}$ & \parbox[t]{7.6cm}{Inflow of asset $n$ in scenario $\omega$ at time $t$ (m$^3$/s)}\\
$q^{\mathrm{out}}_{n,\omega,t}$ & \parbox[t]{7.6cm}{Outflow of asset $n$ in scenario $\omega$ at time $t$ (m$^3$/s)}\\
$q^{\mathrm{tr}}_{n, \omega, t}$ & \parbox[t]{7.6cm}{Outflow of turbines $n$ in scenario $\omega$ at time $t$ (m$^3$/s)}\\
$z^{\mathrm{fbl}}_{n, \omega, t}$ & \parbox[t]{7.6cm}{Forebay level of asset $n$ in scenario $\omega$ at time $t$ (m)}
\end{tabular}

\noindent \begin{tabular}{@{}l@{\hspace{0.2cm}}l@{}}
$\delta^{\mathrm{E}\uparrow}_{\omega, t}$ & \parbox[t]{7.6cm}{Positive imbalance in scenario $\omega$ at time $t$ (MWh)}\\
$\delta^{\mathrm{E}\downarrow}_{\omega, t}$ & \parbox[t]{7.6cm}{Negative imbalance in scenario $\omega$ at time $t$ (MWh)}\\
$\boldsymbol{\lambda}_{n,\omega}$ & \parbox[t]{7.6cm}{Vector of dual variables corresponding to the consensus constraints of asset $n$ in scenario $\omega$}
\end{tabular}

\section{Introduction}
\subsection{Literature Review}
\IEEEPARstart{A}{s} the energy sector shifts from conventional generators to variable renewable energy sources (vRES) such as wind and solar, optimizing the bidding strategy and operation of these stochastic resources by means of controllable units has gained significant attention \cite{Koraki}. Hydropower plays a critical role in this transition, providing nearly half of the world's low-carbon electricity and offering an unparalleled 1500 TWh of storage capacity, dwarfing global battery storage by a factor of 2200 in 2021 \cite{Hydropower_Special_Market_Report}. Among hydropower systems, cascaded hydropower plants (CHPP) are effective in enhancing vRES market value \cite{Zhang1} and supporting grid operations \cite{Zhao}. However, co-optimizing the market-based operation of CHPP and vRES is challenging due to uncertainty sources, complex spatiotemporal constraints, and complicated nonlinear dynamics.

Approaches for managing uncertainty in CHPP-vRES systems can be broadly categorized into robust optimization \cite{Apostolopoulou1}, interval optimization \cite{Zhou2}, chance-constrained programming \cite{Qiu}, and scenario-based stochastic programming \cite{Fleten}. A comprehensive analysis of these methods within the broader context of power system optimization under uncertainty is detailed in \cite{Roald}.
Among these, scenario-based stochastic programming is often preferred for its flexibility in modeling uncertainty by incorporating multiple potential outcomes, or \textit{scenarios}, each weighted by its associated probability. Recent studies have shown the effectiveness of this approach in managing multiple sources of uncertainty inherent in the operation of CHPP, with applications including the optimization of market participation strategies \cite{LI2023109379}, maintenance scheduling \cite{ZHONG2024179}, and day-ahead operational planning \cite{LIU2023109327}. Scenario-based stochastic programming is particularly well-suited for handling unbounded uncertainty sources such as prices \cite{Helseth1}, inflow \cite{Wang}, and demand \cite{Shi}, which cannot typically be confined a priori to a specific set or bounded range, as required in robust and interval optimization approaches, respectively. Moreover, the scenario-based stochastic programming modeling approach yields solutions tailored to expected outcomes \cite{Zhang2}, in contrast to approaches like interval optimization, which typically focus on hedging against worst-case realizations \cite{Zhou2}.
However, it can be computationally inefficient, particularly with large sets of scenarios. In bidding problems, scenario-based approaches allow for the use of price-quantity bidding curves, enabling energy bid adjustments based on predicted prices and generation levels \cite{Mazzi}. This is particularly beneficial for vRES paired with flexible resources such as hydropower units \cite{Karasavvidis}.

Besides uncertainty management, a key challenge in CHPP operation is balancing modeling accuracy with tractability. Comprehensive physical or operational models result in complex mixed-integer nonlinear programming (MINLP) problems \cite{Yu}, for which there is still no efficient off-the-shelf solver. Significant complexity arises from the coupling among units, and the complicated nonlinear dynamics involved \cite{Séguin}. To address this, the original model is often simplified using three main methods \cite{Shayesteh}, namely aggregation, optimization/heuristic, and decomposition methods. Aggregation methods have been extensively applied to both short-term \cite{Guedes} and long-term \cite{Helseth2} scheduling of CHPP, but are best suited to systems with similar reservoir and inflow characteristics, as the individual plant constraints are overlooked. Researchers have also investigated ways to simplify the MINLP model while accounting for the individual dynamics of each plant. Solving these optimization models directly with commercial solvers is often very challenging. A common method extensively used in the current literature is outer approximation \cite{López-Salgado}, which decomposes the MINLP problem into a mixed-integer linear programming (MILP) master problem and a set of nonlinear programming (NLP) subproblems, although this often results in a suboptimal performance. Alternatively, the MINLP problem can be simplified in various manners, reducing to linear programming (LP) \cite{Apostolopoulou2}, MILP \cite{Lu}, quadratic programming \cite{Finardi}, mixed-integer quadratic programming (MIQP) \cite{Takriti} or NLP \cite{Huang} problems. MILP approximations are increasingly popular, offering a balance between modeling accuracy and tractability \cite{Tong}. To solve these models, methods such as linear decision rules \cite{Egging} and dynamic programming \cite{Huang} have been applied. While rule-based approaches may oversimplify the problem, dynamic programming can be computationally expensive due to dimensionality issues.

To overcome these limitations, researchers have explored decomposition techniques. Two main strategies identified in the literature are scenario decomposition and spatial decomposition \cite{Scuzziato}. Scenario decomposition methods, such as Benders decomposition \cite{López-Salgado} or the consensus alternating direction method of multipliers (ADMM) \cite{Takriti}, split the stochastic problem into multiple deterministic subproblems, typically one per scenario. Spatial decomposition, often achieved using Lagrangian relaxation techniques like dual decomposition \cite{Helseth3}, ADMM \cite{Ma}, or the auxiliary problem principle (APP) \cite{Santosuosso2}, breaks the problem down by units or groups of units. Both strategies can reduce computational burden compared to the centralized (non-decomposed) counterpart, however they often require assumptions of convexity to guarantee converging to the global solution. Even for MILP approximations, applying these methods to CHPP scheduling results in heuristic approaches \cite{Shayesteh}.

\subsection{Motivation and Contributions}
Co-optimized bidding strategy for CHPP and vRES is a complex stochastic MINLP problem. While MILP approximations have become popular for their ability to balance modeling accuracy versus computational tractability, they still incur significant computational burdens. Decomposition approaches have been effective in alleviating such a burden, however existing works focus on either spatial or scenario decomposition, addressing only one dimension of the problem's complexity, and most importantly, do not offer any solution guarantee in terms of optimality. To date, no decomposition method for CHPP-vRES bidding effectively scales both spatial and scenario dimensions while providing a performance guarantee.

This paper develops a scalable co-optimization method for the joint bidding of cascaded hydropower, wind and solar energy units in the day-ahead market, accounting for uncertainties in water inflows, prices, and renewable power generation. A MILP formulation is provided, which is then decomposed using a consensus ADMM approach. Since a straightforward application of ADMM to a mixed-integer problem inherently results in a heuristic method, the framework is extended to incorporate \textit{trustworthy} upper and lower bounds on the distributed solution. This enhancement not only accelerates the algorithm but also facilitates the assessment of solution quality, ensuring that the solution is $\epsilon$-suboptimal, meaning its objective value is guaranteed to be at most $\epsilon$ (e.g., 0.01\%) worse than the global optimal solution. In other words, the proposed distributed algorithm offers a \textit{performance guarantee}, as it provides solutions that operators can confidently rely on. Finally, drawing on recent advancements in consensus ADMM \cite{Gade}, the existence and validity of the bounds is demonstrated. Table~\ref{tab:literature_comparison} compares this study with the relevant related literature.

The \textbf{key contributions of this paper} are summarized as follows. First, a scenario-spatial decomposition of the centralized stochastic CHPP-vRES bidding strategy is proposed. This involves decomposing the problem with respect to three sets of coupling constraints: consistency constraints across price, inflow, and vRES scenarios (scenario coupling); power balancing among the aggregated energy units (spatial coupling); and hydraulic links within the hydropower cascade (spatial coupling). This full-scale decomposition distinguishes the proposed approach from previous studies that either address the centralized problem directly, e.g., \cite{Zhao, Apostolopoulou1, Qiu, Egging}, or decouple only one set of constraints, e.g., \cite{Santosuosso2, López-Salgado, Takriti, Helseth3, Ma}.  Second, a consensus ADMM-based distributed algorithm is developed to derive trustworthy upper and lower bounds on the global optimum of the bidding problem. As a distinct feature, this method offers a performance guarantee, even for MILP problems, by using bounds to quantify the ``distance" between the distributed solution and the global optimum. In contrast, existing decomposition approaches typically yield heuristic solutions when applied to mixed-integer models, e.g., \cite{Santosuosso2, López-Salgado, Takriti, Helseth3, Ma}. Finally, the paper offers a mathematical proof demonstrating the existence and validity of the bounds.

\begin{table*}[!t]
\caption{Comparison of this study with the relevant related literature \label{tab:literature_comparison}}
\centering
\begin{tabular}{|c|c|c|c|c|c|c|}
\hline
\textbf{Ref.} & \textbf{Computational strategy} & \textbf{Uncertainty modeled?} & \textbf{Problem} & \textbf{Solution method} & \textbf{Decomposition} & \textbf{Trustworthy}\\
\hline
\cite{Guedes} & Aggregation & $\times$ & MILP & Branch and bound & $\times$ & \checkmark\\
\cite{Shayesteh} & Aggregation & \checkmark (stochastic programming) & NLP & NLP solver & $\times$ & \checkmark\\
\cite{Huang} & Optimization & $\times$ & NLP & Dynamic programming & $\times$ & \checkmark\\
\cite{Apostolopoulou1} & Optimization & \checkmark (robust optimization) & LP & LP solver & $\times$ & \checkmark\\
\cite{Lu} & Optimization & \checkmark (interval programming) & MILP & Branch and bound & $\times$ & \checkmark\\
\cite{Wang} & Optimization & \checkmark (stochastic programming) & MILP & Branch and bound & $\times$ & \checkmark\\
\cite{Shi} & Optimization & \checkmark (stochastic programming) & MILP & Branch and bound & $\times$ & \checkmark\\
\cite{LIU2023109327} & Optimization & \checkmark (stochastic programming) & MILP & Branch and bound & $\times$ & \checkmark\\
\cite{LI2023109379} & Optimization & \checkmark (stochastic programming) & MILP & Dynamic programming & $\times$ & \checkmark\\
\cite{Egging}  & Optimization & \checkmark (stochastic programming) & LP & Linear decision rules & $\times$ & \checkmark\\
\cite{Santosuosso2} & Decomposition & $\times$ & MILP & APP & Spatial & $\times$\\
\cite{Ma} & Decomposition & $\times$ & MIQP & ADMM & Spatial & $\times$\\
\cite{Helseth3} & Decomposition & \checkmark (stochastic programming) & MILP & Dual decomposition & Spatial & $\times$\\
\cite{ZHONG2024179} & Decomposition & \checkmark (stochastic programming) & MILP & Benders decomposition & Scenario & $\times$\\
\cite{López-Salgado} & Decomposition & \checkmark (stochastic programming) & MINLP & Benders decomposition & Scenario & $\times$\\
\cite{Takriti} & Decomposition & \checkmark (stochastic programming) & MIQP & Consensus ADMM & Scenario & $\times$\\
This paper & Decomposition & \checkmark (stochastic programming) & MILP & Consensus ADMM & Spatial \& scenario & \checkmark\\
\hline
\end{tabular}
\end{table*}

The proposed approach is validated through a case study simulating the real-world system managed by the French aggregator Compagnie Nationale du Rhône \cite{Piron}. The model integrates real operational curves to account for the individual safety constraints of each hydropower plant. The non-convex nature of these curves introduces significant complexity into the optimization problem. Furthermore, dynamic price-quantity bidding curves are derived to enhance the adaptability of the bidding strategy to price forecasts, providing greater flexibility compared to fixed hourly bids, which are characterized by a uniform quantity offered for all price levels.  Simulation results assess the modeling accuracy, the aggregator's performance in the market, and the scalability of the distributed method. An ex-post out-of-sample validation is performed to determine the minimum number of scenarios necessary to adequately represent the underlying uncertainties.

The remainder of the article is organized as follows: Section \ref{sec:methodology} details the problem and its decomposition, Section \ref{sec:results} discusses the results, and Section \ref{sec:conclusions} concludes the study.

\section{Methodology}
\label{sec:methodology}
This section formulates the problem and outlines the proposed distributed method with trustworthy bounds. Let $n \in \boldsymbol{N}$ denote a hydropower plant within the cascaded system. The problem is formulated in discrete time over a finite set of sampling times $t \in \boldsymbol{T}$, with a sampling period of $\Delta_T$. The scenario $\omega \in \boldsymbol{\Omega}$ denotes a potential realization of uncertainty. The uncertain parameters are denoted by the hat symbol $\hat{(.)}$.
The symbol $\coloneqq$ denotes the definition of variables, vectors, functions, and sets, while $|\cdot|$ denotes the cardinality of a set.

\subsection{Problem Statement}

\begin{figure}[!t]
\centering
\includegraphics[scale=0.082]{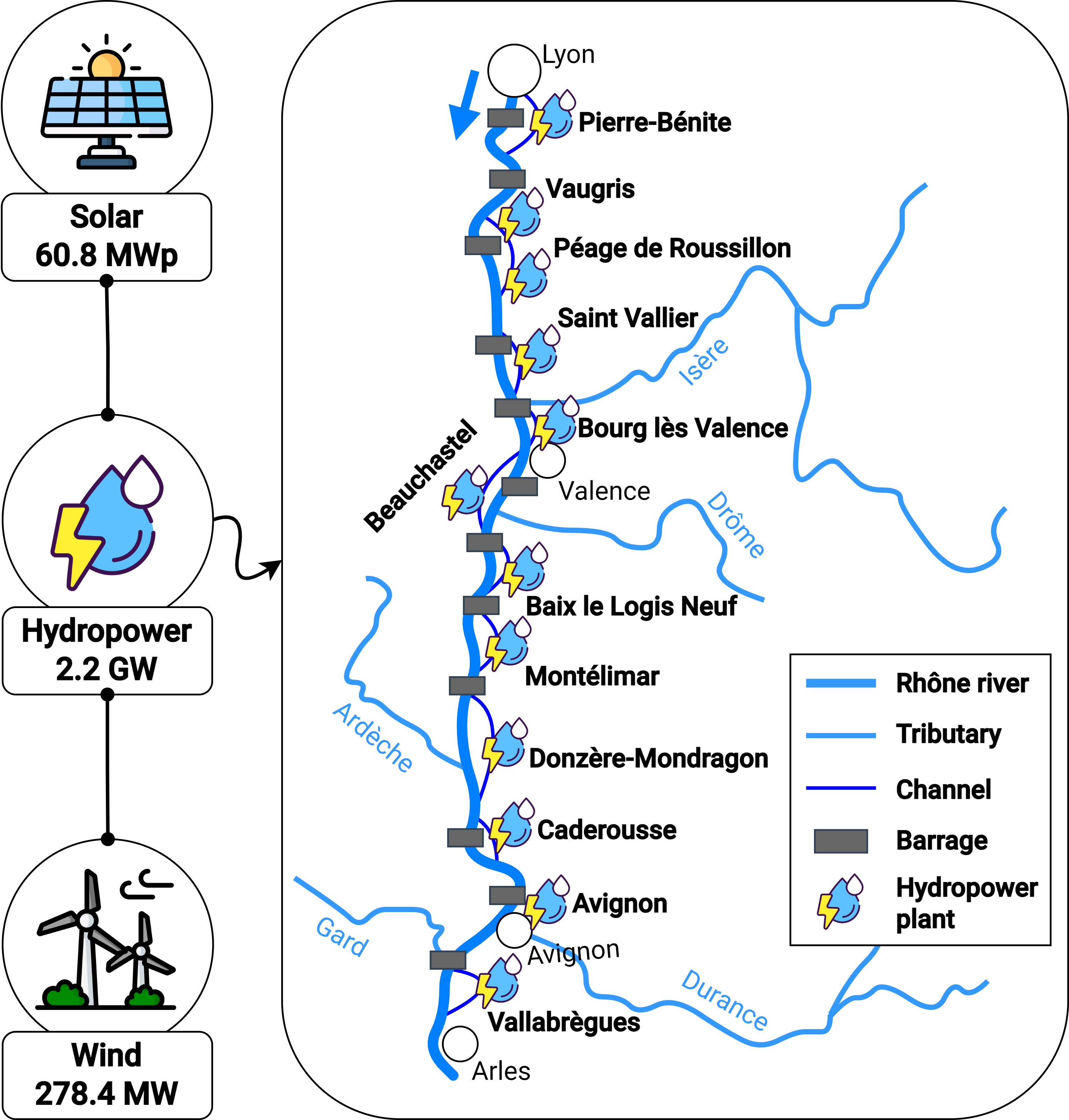}
\caption{Diagram of the VPP under consideration, comprising CHPP and vRES.}
\label{fig:vpp}
\end{figure}
The goal is to optimize the participation of a price-taking energy producer in the day-ahead market, known as the \textit{bidding problem} \cite{Steeger}. The system mimics a real-world virtual power plant (VPP) operated by Compagnie Nationale du Rhône, comprising 12 CHPP (totaling 2.2 GW of installed capacity), 6 wind farms (278.4 MW), and 4 solar farms (60.8 MW), as shown in Fig.~\ref{fig:vpp}. The study considers three sources of uncertainty—water inflow, vRES output, and market prices—which are modeled using a finite set of scenarios.

The energy volumes offered to the market are denoted by $e_{\omega,t}$, and the day-ahead energy price forecast by $\hat{\pi}^{\mathrm{E}}$. A dual pricing settlement scheme is employed, with separate prices for positive imbalances, $\delta^{\mathrm{E}\uparrow}$ (energy injection falling short of the offer), and negative imbalances, $\delta^{\mathrm{E}\downarrow}$ (energy injection exceeding the offer). The forecasted penalty prices for these imbalances are denoted by $\hat{\pi}^{\mathrm{E}\uparrow}$ and $\hat{\pi}^{\mathrm{E}\downarrow}$, respectively. Bids are submitted to the market operator as price-quantity pairs, forming a set of price-quantity bidding curves \cite{Mazzi}. The optimization problem considers all scenario prices as potential bidding points, with the energy produced for each hour and price determining the bid.
The use of price scenarios as potential bidding points has been extensively documented in the literature on optimal bidding strategies for VPPs \cite{7762908}, particularly within the framework of dual pricing settlement schemes \cite{8584080}. For a comprehensive comparison of the price-quantity bidding strategy with alternative price-only and quantity-only strategies, the reader is referred to \cite{Karasavvidis}.
Consistent with common rules in many European markets \cite{Steeger}, all supply bids within an hour must be non-decreasing as the price increases.

\subsection{Centralized Optimization Problem}
The VPP in Fig.~\ref{fig:vpp} is modeled with three key assumptions: (i) vRES are uncontrollable, facilitating the assessment of the cascade's ability to manage their variability; (ii) tailrace water levels remain constant due to the minimal short-term impact of run-of-the-river plant discharges; and (iii) all turbines within each hydropower plant are of the same type.
In accordance with assumption (iii), the formulation presumes that the turbines are homogeneous within each individual hydropower plant, while allowing for heterogeneous turbines across different plants within the cascaded system. This assumption is adopted to streamline the subsequent discussion and notation. Nevertheless, the proposed formulation can be readily extended to accommodate heterogeneous turbines within a single plant, for example by following the modeling approach presented in \cite{anagnostopoulos2007optimal}.

The dynamics of the forebay reservoir water levels, denoted by $z^{\mathrm{fbl}}$, are computed as a function of the reservoir's surface area $S$, inflow $q^{\mathrm{in}}$, outflow $q^{\mathrm{out}}$, and the initial forebay water level $Z^{\mathrm{fbl},0}$. For each asset pair, a specific point between the two reservoirs, termed the \textit{control point}, is selected, and the forebay water level is derived from the volume balance at this control point, leading to the following state-space model:
\begin{subequations}
\label{eq:z_fbl_dynamics}
\begin{align}
z^{\mathrm{fbl}}_{n,\omega,t} & = z^{\mathrm{fbl}}_{n,\omega,t-1} + \frac{\left(q^{\mathrm{in}}_{n,\omega,t} - q^{\mathrm{out}}_{n,\omega,t}\right) \Delta_T}{S_{n,\omega,t}}, \; \forall n, \forall \omega, \forall t,\\
z^{\mathrm{fbl}}_{n,\omega,0} & = z^{\mathrm{fbl}}_{n,\omega,|\boldsymbol{T}|} = Z^{\mathrm{fbl},0}_n, \; \forall n, \forall \omega.
\end{align}
\end{subequations}
The inflow into each reservoir is composed of the water discharged from the upstream hydropower system through the barrage and turbines, denoted by $q^{\mathrm{br}}$ and $q^{\mathrm{tr}}$, respectively, and the external inflow, denoted by $\hat{Q}^{\mathrm{ext}}$. Constraints \eqref{eq:q_in} describe the inflow into the $n$-th reservoir under scenario $\omega$ at time $t$:
\begin{multline}
\label{eq:q_in}
q^{\mathrm{in}}_{n,\omega,t} = q^{\mathrm{br}}_{n-1,\omega,t - \tau^{\mathrm{br}}_{n-1,n}} + q^{\mathrm{tr}}_{n-1,\omega,t - \tau^{\mathrm{tr}}_{n-1,n}} + \hat{Q}^{\mathrm{ext}}_{n,\omega,t},\\ \; \forall n, \forall \omega, \forall t,
\end{multline}
where $\tau_{n-1,n}^{\mathrm{tr}}$ and $\tau_{n-1,n}^{\mathrm{br}}$ represent the water propagation time from the plant $n-1$ to the control point of plant $n$ via turbines and barrage, respectively.
Similarly, the outflow is given by
\begin{equation}
\label{eq:q_out}
q^{\mathrm{out}}_{n,\omega,t} = q^{\mathrm{br}}_{n,\omega,t} + q^{\mathrm{tr}}_{n,\omega,t}, \; \forall n, \forall \omega, \forall t.
\end{equation}

To ensure stability and safety, CHPP operate under strict operational curves. Fig.~\ref{fig:operational_curve} shows an example of an operational curve for a hydropower plant. These curves are developed based on approved guidelines from relevant authorities and take into account various constraints, including navigation, irrigation, nuclear safety, agriculture, and tourism. During \textit{low-flow periods}, water levels are kept high for navigation, while during \textit{floods}, levels are controlled to mimic pre-barrage conditions. Reservoirs provide storage capacity only during normal inflow conditions, i.e., during \textit{energetic periods}.

\begin{figure}[!t]
\centering
\includegraphics[scale=.48]{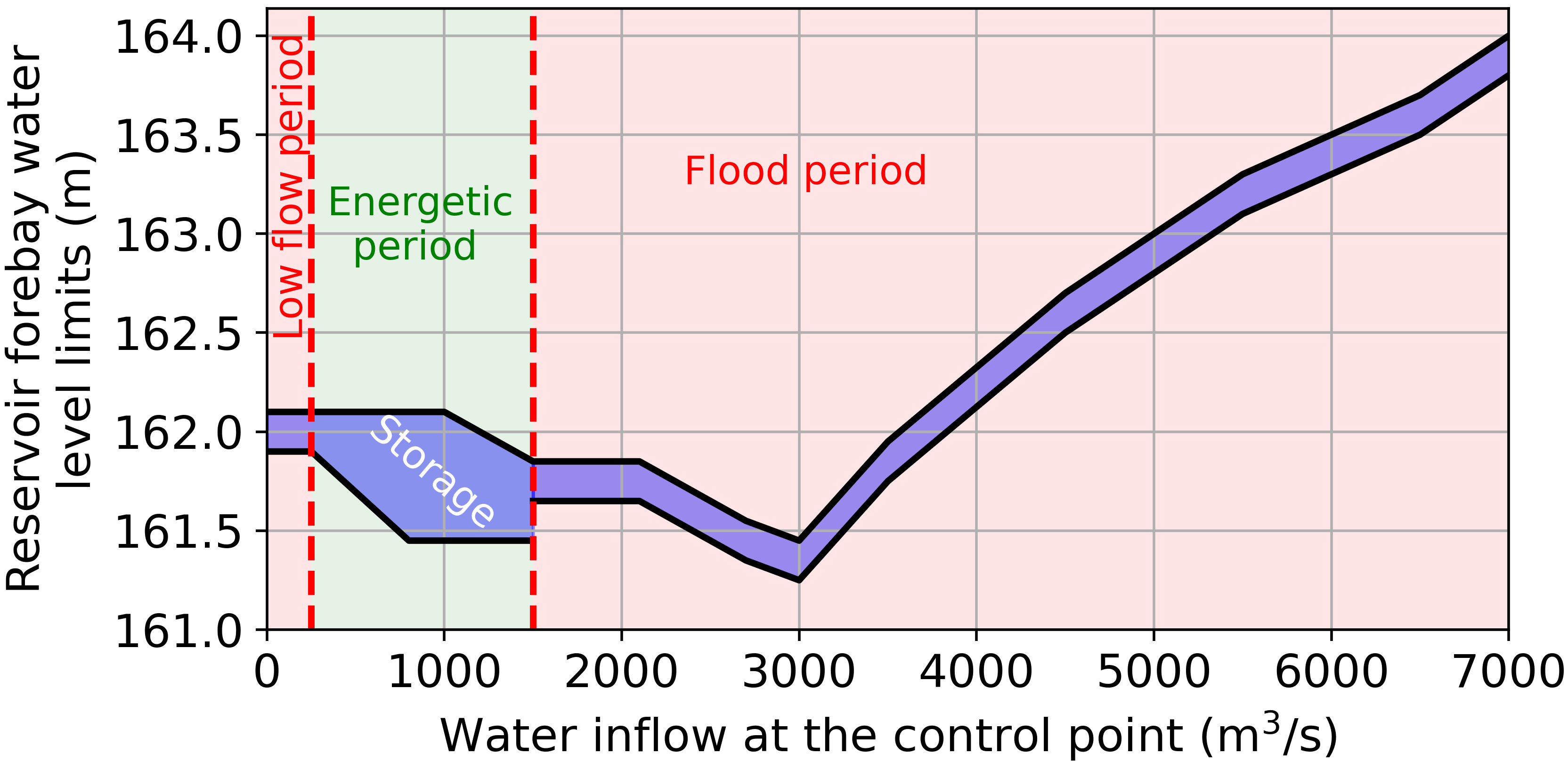}
    \caption{Example of an operational curve from Compagnie Nationale du Rhône.}
    \label{fig:operational_curve}
\end{figure}

The operational curves are approximated using the piecewise linear constraints in \eqref{eq:operational_curves_approx}. This involves dividing the inflow range into segments $Q^{\mathrm{in}}$, each linked to fixed forebay water levels $\underline{Z}^{\mathrm{fbl}}$ and $\overline{Z}^{\mathrm{fbl}}$. A set of inflow segments $\boldsymbol{S}_{n}^{\boldsymbol{\mathrm{oc}}}$ is computed for each hydropower plant $n$:
\begin{subequations}\label{eq:operational_curves_approx}
\begin{align}
q^{\mathrm{in}}_{n,\omega,t} &\geq b^{\mathrm{oc}}_{n,\omega,t,i} Q^{\mathrm{in}}_{n,i},\; \forall i \in \boldsymbol{S}_{n}^{\boldsymbol{\mathrm{oc}}}, \forall n, \forall \omega, \forall t,\\
q^{\mathrm{in}}_{n,\omega,t} &< b^{\mathrm{oc}}_{n,\omega,t,i} Q^{\mathrm{in}}_{n,i+1} + \left(1 - b^{\mathrm{oc}}_{n,\omega,t,i}\right) M^{\mathrm{oc}},\notag \\
&\hspace{10em} \forall i \in \boldsymbol{S}_{n}^{\boldsymbol{\mathrm{oc}}}, \forall n, \forall \omega, \forall t.
\end{align}
\end{subequations}
Here, $M^{\mathrm{oc}}$ is a large positive constant, and the binary variables $b^{\mathrm{oc}}$ are constrained by
\begin{equation}
\label{eq:operational_curves_binaries}
\sum_{i \in \boldsymbol{S}^{\boldsymbol{\mathrm{oc}}}_n} b^{\mathrm{oc}}_{n,\omega,t,i} = 1, \; \forall n, \forall \omega, \forall t.
\end{equation}
Constraints \eqref{eq:z_fbl_limits} enforce the forebay water level limits:
\begin{multline}
\label{eq:z_fbl_limits}
\sum_{i \in \boldsymbol{S}^{\boldsymbol{\mathrm{oc}}}_n} \underline{Z}^{\mathrm{fbl}}_{n,i} b^{\mathrm{oc}}_{n,\omega,t,i} \leq z^{\mathrm{fbl}}_{n,\omega,t} \leq \sum_{i \in \boldsymbol{S}^{\boldsymbol{\mathrm{oc}}}_n} \overline{Z}^{\mathrm{fbl}}_{n,i} b^{\mathrm{oc}}_{n,\omega,t,i},\\ \; \forall n, \forall \omega, \forall t.
\end{multline}
The turbines are constrained by the ramp limit:
\begin{subequations}
\label{eq:q_tr_maping_limit}
\begin{align}
q^{\mathrm{tr}}_{n,\omega,t} - q^{\mathrm{tr}}_{n,\omega,t-1} \leq \Delta q^{\mathrm{tr}}_n, \; \forall n, \forall \omega, \forall t,\\
q^{\mathrm{tr}}_{n,\omega,t-1} - q^{\mathrm{tr}}_{n,\omega,t} \leq \Delta q^{\mathrm{tr}}_n, \; \forall n, \forall \omega, \forall t.
\end{align}
\end{subequations}
For safety reasons, the opening of the barrage is allowed exclusively when the reservoir is at full capacity:
\begin{equation}
\label{eq:q_br_safety1}
b^{\mathrm{br}}_{n,\omega,t} \underline{Q}^{\mathrm{br}}_{n} \leq q^{\mathrm{br}}_{n,\omega,t} \leq b^{\mathrm{br}}_{n,\omega,t} M^{\mathrm{br}}, \; \forall n, \forall \omega, \forall t.
\end{equation}
Here, $\underline{Q}^{\mathrm{br}}$ is the minimum barrage discharge, $M^{\mathrm{br}}$ is a large positive constant, and the binary variables $b^{\mathrm{br}}$ satisfy
\begin{equation}
\label{eq:q_br_safety2}
b^{\mathrm{br}}_{n,\omega,t} \leq 1 - \frac{\sum_{i \in \boldsymbol{S}^{\boldsymbol{\mathrm{oc}}}_n} \overline{Z}^{\mathrm{fbl}}_{n,i} b^{\mathrm{oc}}_{n,\omega,t,i} - z^{\mathrm{fbl}}_{n,\omega,t}}{\mathrm{max}\left(\left\{\overline{Z}^{\mathrm{fbl}}_{n,i}\right\}_{i \in \boldsymbol{S}^{\boldsymbol{\mathrm{oc}}}_n}\right)},\; \forall n, \forall \omega, \forall t.
\end{equation}

Similarly, when the turbines operate, their discharge is constrained by
\begin{equation}
\label{eq:q_tr_limits}
\underline{Q}^{\mathrm{tr}}_{n} \leq q^{\mathrm{tr}}_{n,\omega,t} \leq \overline{Q}^{\mathrm{tr}}_{n}, \; \forall n, \forall \omega, \forall t.
\end{equation}

The hydropower output, $p^{\mathrm{H}}$, can be calculated as a function of the turbines outflow, net hydraulic head $h$, and turbine efficiency $\eta^{\mathrm{tr}}$, as described by
\begin{equation}
\label{eq:NL_hydropower_output}
p^{\mathrm{H}}_{n,\omega,t} = w \, g \, \eta^{\mathrm{tr}}_n \, q^{\mathrm{tr}}_{n,\omega,t} \, h_{n,\omega,t} \, 10^{-6}, \; \forall n, \forall \omega, \forall t,
\end{equation}
where $g$ is the acceleration due to gravity and $w$ is the density of water. The factor $10^{-6}$ is used to convert the power output from watts to megawatts. The head is defined by
\begin{equation}
\label{eq:hydraulic_head}
h_{n,\omega,t} = z^{\mathrm{fbl}}_{n,\omega,t} - Z^{\mathrm{tlr}}_n, \; \forall n, \forall \omega, \forall t.
\end{equation}
Here, $Z^{\mathrm{tlr}}$ denotes the constant tailrace reservoir water level. The hydropower generation function exhibits nonlinearity due to its bilinear dependence on $q^{\mathrm{tr}}$ and $h$. To address this, the McCormick approximation \cite{mccormick1976computability} in \eqref{eq:hydropower_output_McCormick} is employed to replace the bilinear term in \eqref{eq:NL_hydropower_output} with a convex envelope:
\begin{subequations}
\label{eq:hydropower_output_McCormick}
\begin{align}
p^{\mathrm{H}}_{n,\omega,t} & \geq C_n \left(\underline{Q}^{\mathrm{tr}}_{n} h_{n,\omega,t} + \underline{H}_{n} q^{\mathrm{tr}}_{n,\omega,t} - \underline{Q}^{\mathrm{tr}}_{n} \underline{H}_{n}\right),\\
p^{\mathrm{H}}_{n,\omega,t} & \geq C_n \left(\overline{Q}^{\mathrm{tr}}_{n} h_{n,\omega,t} + \overline{H}_{n} q^{\mathrm{tr}}_{n,\omega,t} - \overline{Q}^{\mathrm{tr}}_{n}  \overline{H}_{n}\right),\\
p^{\mathrm{H}}_{n,\omega,t} & \leq C_n \left(\underline{Q}^{\mathrm{tr}}_{n} h_{n,\omega,t} + \overline{H}_{n} q^{\mathrm{tr}}_{n,\omega,t} - \underline{Q}^{\mathrm{tr}}_n \overline{H}_{n}\right),\\
p^{\mathrm{H}}_{n,\omega,t} & \leq C_n \left(\overline{Q}^{\mathrm{tr}}_{n} h_{n,\omega,t} + \underline{H}_{n} q^{\mathrm{tr}}_{n,\omega,t} - \overline{Q}^{\mathrm{tr}}_n \underline{H}_{n}\right).
\end{align}
\end{subequations}
Here, $C_n = 10^{-6} w \, g \, \eta^{\mathrm{tr}}_n$, while $\overline{H}$ and $\underline{H}$ denote the maximum and minimum hydraulic head, respectively.
Fig.~\ref{fig:McCormick_approximation_hydropower_gen_fun} illustrates the McCormick approximation in \eqref{eq:hydropower_output_McCormick}. This linearization is a standard method for constructing convex envelopes around bilinear functions \cite{najman2019tightness}, with broad applications in hydropower modeling \cite{Shi}. It is considered one of the most accurate and tight approximations available for bilinear terms \cite{blom2024single}. A detailed error analysis is provided in \cite{flamm2020two}, showing that the largest deviations occur at the intersections of the bounding hyperplanes defined by \eqref{eq:hydropower_output_McCormick}.
For instance, assuming a turbines discharge in the range $[0,1000]$ m$^3$/s, maximum head of $5$ m, $95$\% turbine efficiency, water density of $1000$ kg/m$^3$, and gravitational acceleration of $9.81$ m/s$^2$, this approximation yields a maximum error of $\sim3$\% of the generation capacity.

\begin{figure}[!t]
\centering
\includegraphics[scale=.8]{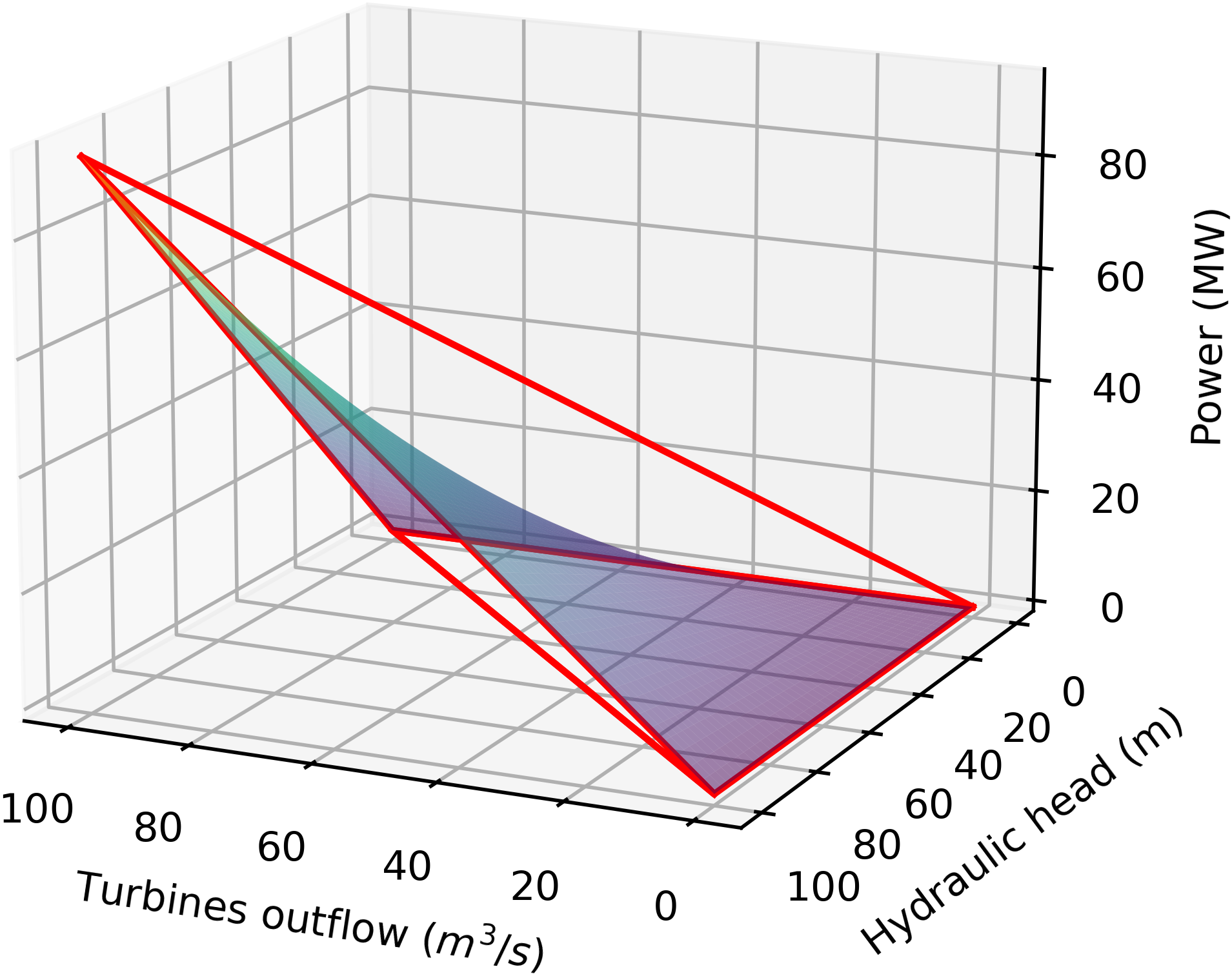}
    \caption{Visualization of the McCormick approximation in \eqref{eq:hydropower_output_McCormick}, assuming a turbine efficiency of 95\%.}
    \label{fig:McCormick_approximation_hydropower_gen_fun}
\end{figure}

The range of admissible values for $p^{\mathrm{H}}$ is enforced by
\begin{equation}
\label{eq:p_H_limits}
0 \leq p^{\mathrm{H}}_{n,\omega,t} \leq \overline{P}^{\mathrm{H}}_{n,\omega,t}, \; \forall n, \forall \omega, \forall t.
\end{equation}

Constraints \eqref{eq:power_balancing} define the VPP's power balance as
\begin{multline}
\label{eq:power_balancing}
\left( \sum_{n \in \boldsymbol{N}} p^{\mathrm{H}}_{n,\omega,t} + \hat{P}^{\mathrm{W}}_{\omega,t} + \hat{P}^{\mathrm{S}}_{\omega,t} \right) \Delta_T - e_{\omega, t} - \delta^{\mathrm{E}\downarrow}_{\omega,t}\\ + \delta^{\mathrm{E}\uparrow}_{\omega,t} = 0, \; \forall \omega, \forall t,
\end{multline}
where $\hat{P}^\mathrm{W}$ and $\hat{P}^\mathrm{S}$ denote the total wind and solar power generation of the VPP, respectively.

The aggregator minimizes the objective function defined as
\begin{equation}
\label{eq:cen_obj_fun}
J \coloneqq \sum_{\omega \in \boldsymbol{\Omega}} \sum_{t \in \boldsymbol{T}}  P_{\omega} \left(\hat{\pi}^{\mathrm{E}\uparrow}_{\omega,t} \delta^{\mathrm{E}\uparrow}_{\omega,t} - \hat{\pi}^{\mathrm{E}\downarrow}_{\omega,t} \delta^{\mathrm{E}\downarrow}_{\omega,t} - \hat{\pi}^{\mathrm{E}}_{\omega,t} e_{\omega, t}\right),
\end{equation}
representing the cost incurred by the VPP in the day-ahead market. Minimizing this cost is equivalent to maximizing the VPP's profit. Here, $P_\omega$ is the probability of scenario $\omega$. Constraints \eqref{eq:price_quantity_bidding_curves_consistency} ensure price-quantity bidding curves taking a non-decreasing form:
\begin{multline}
\label{eq:price_quantity_bidding_curves_consistency}
\left(\hat{\pi}^{\mathrm{E}}_{\omega,t} - \hat{\pi}^{\mathrm{E}}_{\omega',t}\right) \left(e_{\omega,t} - e_{\omega', t}\right) \geq 0,\; \forall \omega' \in \boldsymbol{\Omega} \setminus \{\omega\},\\ \forall \omega, \forall t.
\end{multline}

Finally, the set of decision variables $\boldsymbol{z}$ is defined as
\begin{multline}
\label{eq:decision_variables_set}
\boldsymbol{z} \coloneqq \Bigl\{e_{\omega,t}, \delta^{\mathrm{E}\uparrow}_{\omega,t}, \delta^{\mathrm{E}\downarrow}_{\omega,t}, q^{\mathrm{tr}}_{n,\omega,t}, q^{\mathrm{br}}_{n,\omega,t}, z^{\mathrm{fbl}}_{n,\omega,t}, p^{\mathrm{H}}_{n,\omega,t},\\
b^{\mathrm{br}}_{n,\omega,t}, b^{\mathrm{oc}}_{n,\omega,t,i}\Bigl\}_{i \in \boldsymbol{S}_{n}^{\boldsymbol{\mathrm{oc}}}, n \in \boldsymbol{N}, \omega \in \boldsymbol{\Omega}, t \in \boldsymbol{T}}.
\end{multline}
Then, the MILP given in \eqref{eq:cen_optim_problem} defines the bidding problem:
\begin{subequations}
\label{eq:cen_optim_problem}
\begin{align}
\min_{\boldsymbol{z}} \quad & J\\
\text{subject to:} \quad & \eqref{eq:z_fbl_dynamics}-\eqref{eq:q_tr_limits}, \eqref{eq:hydraulic_head}, \eqref{eq:p_H_limits}, \eqref{eq:power_balancing}, \eqref{eq:price_quantity_bidding_curves_consistency},\\
& \eqref{eq:hydropower_output_McCormick},\; \forall n, \forall \omega, \forall t.
\end{align}
\end{subequations}

\subsection{From Centralized to Distributed Optimization}
The centralized optimization problem \eqref{eq:cen_optim_problem} is characterized by spatial coupling, due to the hydraulic links among CHPP in \eqref{eq:z_fbl_dynamics} and the electrical coupling in \eqref{eq:power_balancing}, as well as the scenario coupling in \eqref{eq:price_quantity_bidding_curves_consistency}. All coupling constraints involve the \textit{global variables} $\boldsymbol{q}^{\boldsymbol{\mathrm{tr}}}_{n,\omega} \coloneqq \left[q^{\mathrm{tr}}_{n,\omega,1}, ..., q^{\mathrm{tr}}_{n,\omega,|\boldsymbol{T}|}\right]^\top$, $\boldsymbol{q}^{\boldsymbol{\mathrm{br}}}_{n,\omega} \coloneqq \left[q^{\mathrm{br}}_{n,\omega,1}, ..., q^{\mathrm{br}}_{n,\omega,|\boldsymbol{T}|}\right]^\top$ and $\boldsymbol{p}^{\boldsymbol{\mathrm{H}}}_{n,\omega} \coloneqq \left[p^{\mathrm{H}}_{n,\omega,1}, ..., p^{\mathrm{H}}_{n,\omega,|\boldsymbol{T}|}\right]^\top$.
By duplicating these variables, the centralized problem \eqref{eq:cen_optim_problem} is reformulated as a consensus problem, which reads as
\begin{subequations}
\label{eq:decomposed_optim_problem}
\begin{align}
\min_{\substack{\boldsymbol{z^{\mathrm{B}}},\\ \left\{\boldsymbol{z}^{\boldsymbol{\mathrm{H}}}_{n,\omega}\right\}_{n \in \boldsymbol{N}, \omega \in \boldsymbol{\Omega}}}} \quad & J\\
    \textrm{subject to} \quad & \boldsymbol{z^{\mathrm{B}}} \in \boldsymbol{\Xi}\left(\boldsymbol{\beta}\right), \label{eq:decomposed_optim_problem_sub1}\\
    & \boldsymbol{z}^{\boldsymbol{\mathrm{H}}}_{n,\omega} \in \boldsymbol{\Gamma}_{n,\omega}\left(\boldsymbol{\theta}_{n,\omega}\right), \; \forall n, \forall \omega, \label{eq:decomposed_optim_problem_sub2}\\
    & \boldsymbol{z}^{\boldsymbol{\mathrm{B},\mathrm{H}}}_{n,\omega} = \boldsymbol{\bar{z}}^{\boldsymbol{\mathrm{B},\mathrm{H}}}_{n,\omega}\; : \; \boldsymbol{\lambda}_{n,\omega}, \; \forall n, \forall \omega. \label{eq:consensus}
\end{align}
\end{subequations}
Here, $\boldsymbol{z^{\mathrm{B}}}$ collects the decision variables associated with the power balancing of the VPP, as defined by
\begin{equation}
\label{eq:z_B_def}
\boldsymbol{z^{\mathrm{B}}} \coloneqq \left\{e_{\omega,t}, \delta^{\mathrm{E}\uparrow}_{\omega,t}, \delta^{\mathrm{E}\downarrow}_{\omega,t}, \boldsymbol{\bar{p}}^{\boldsymbol{\mathrm{H}}}_{n,\omega}\right\}_{n \in \boldsymbol{N}, \omega \in \boldsymbol{\Omega}, t \in \boldsymbol{T}}.
\end{equation}
The decision variables associated with the dynamics of each hydropower plant are collected in $\boldsymbol{z}^{\boldsymbol{\mathrm{H}}}_{n,\omega}$, as defined by
\begin{multline}
\label{eq:z_H_def}
\boldsymbol{z}^{\boldsymbol{\mathrm{H}}}_{n,\omega} \coloneqq \Bigl\{b^{\mathrm{br}}_{n,\omega,t}, b^{\mathrm{oc}}_{n,\omega,t}, h_{n,\omega,t}, z^{\mathrm{fbl}}_{n,\omega,t}, \boldsymbol{\bar{q}}^{\boldsymbol{\mathrm{tr}}}_{n,\omega}, \boldsymbol{\tilde{q}}^{\boldsymbol{\mathrm{tr}}}_{n,\omega}, \boldsymbol{\bar{q}}^{\boldsymbol{\mathrm{br}}}_{n,\omega},\\ \boldsymbol{\tilde{q}}^{\boldsymbol{\mathrm{br}}}_{n,\omega}, \boldsymbol{\tilde{p}}^{\boldsymbol{\mathrm{H}}}_{n,\omega}\Bigl\}_{t \in \boldsymbol{T}}.
\end{multline}
The feasible sets of $\boldsymbol{z}^{\boldsymbol{\mathrm{B}}}$ and $\boldsymbol{z}^{\boldsymbol{\mathrm{H}}}_{n,\omega}$ are denoted by $\boldsymbol{\Xi}\left(\boldsymbol{\beta}\right)$ and $\boldsymbol{\Gamma}_{n,\omega}\left(\boldsymbol{\theta}_{n,\omega}\right)$, respectively. The former is defined by \eqref{eq:p_H_limits}, \eqref{eq:power_balancing} and \eqref{eq:price_quantity_bidding_curves_consistency}. The latter is defined by \eqref{eq:z_fbl_dynamics}-\eqref{eq:q_tr_limits} and \eqref{eq:hydraulic_head}-\eqref{eq:p_H_limits}. Vectors $\boldsymbol{\beta}$ and  $\boldsymbol{\theta}_{n,\omega}$ collect the parameters associated with the constraints defining the feasible sets $\boldsymbol{\Xi}$ and $\boldsymbol{\Gamma}_{n,\omega}$, respectively.

Moreover, the vector $\boldsymbol{z}^{\boldsymbol{\mathrm{B},\mathrm{H}}}_{n,\omega} \in \mathbb{R}^{6 \times |\boldsymbol{T}|}$ collects the global variables of the problem, as defined by
\begin{equation}
\label{eq:z_B_H_def}
\boldsymbol{z}^{\boldsymbol{\mathrm{B},\mathrm{H}}}_{n,\omega} \coloneqq \left[\boldsymbol{q}^{\boldsymbol{\mathrm{tr}}}_{n-1,\omega}, \boldsymbol{q}^{\boldsymbol{\mathrm{br}}}_{n-1,\omega}, \boldsymbol{q}^{\boldsymbol{\mathrm{tr}}}_{n,\omega}, \boldsymbol{q}^{\boldsymbol{\mathrm{br}}}_{n,\omega}, \boldsymbol{p}^{\boldsymbol{\mathrm{H}}}_{n,\omega}, \boldsymbol{p}^{\boldsymbol{\mathrm{H}}}_{n,\omega}\right]^\top.
\end{equation}
The vector $\boldsymbol{\bar{z}}^{\boldsymbol{\mathrm{B},\mathrm{H}}}_{n,\omega} \in \mathbb{R}^{6 \times |\boldsymbol{T}|}$ collects the \textit{local copies} of the global variables in $\boldsymbol{z}^{\boldsymbol{\mathrm{B},\mathrm{H}}}_{n,\omega}$, as defined by
\begin{equation}
\label{eq:z_B_H_bar_def}
\boldsymbol{\bar{z}}^{\boldsymbol{\mathrm{B},\mathrm{H}}}_{n,\omega} \coloneqq \left[\boldsymbol{\tilde{q}}^{\boldsymbol{\mathrm{tr}}}_{n-1,\omega}, \boldsymbol{\tilde{q}}^{\boldsymbol{\mathrm{br}}}_{n-1,\omega}, \boldsymbol{\bar{q}}^{\boldsymbol{\mathrm{tr}}}_{n,\omega}, \boldsymbol{\bar{q}}^{\boldsymbol{\mathrm{br}}}_{n,\omega}, \boldsymbol{\tilde{p}}^{\boldsymbol{\mathrm{H}}}_{n,\omega}, \boldsymbol{\bar{p}}^{\boldsymbol{\mathrm{H}}}_{n,\omega}\right]^\top.
\end{equation}
Consistency between global variables and local copies is enforced through the \textit{consensus constraints} in \eqref{eq:consensus}. The corresponding vector of dual variables $\boldsymbol{\lambda}_{n,\omega} \in \mathbb{R}^{6 \times |\boldsymbol{T}|}$ is defined as
\begin{equation}
\label{eq:lambda_def}
\boldsymbol{\lambda}_{n,\omega} \coloneqq \left[\boldsymbol{\tilde{\lambda}}^{\boldsymbol{\mathrm{tr}}}_{n-1,\omega}, \boldsymbol{\tilde{\lambda}}^{\boldsymbol{\mathrm{br}}}_{n-1,\omega}, \boldsymbol{\bar{\lambda}}^{\boldsymbol{\mathrm{tr}}}_{n,\omega}, \boldsymbol{\bar{\lambda}}^{\boldsymbol{\mathrm{br}}}_{n,\omega}, \boldsymbol{\tilde{\lambda}}^{\boldsymbol{\mathrm{p}}}_{n,\omega}, \boldsymbol{\bar{\lambda}}^{\boldsymbol{\mathrm{p}}}_{n,\omega}\right]^\top.
\end{equation}

\begin{figure}[!t]
\centering
\includegraphics[scale=.049]{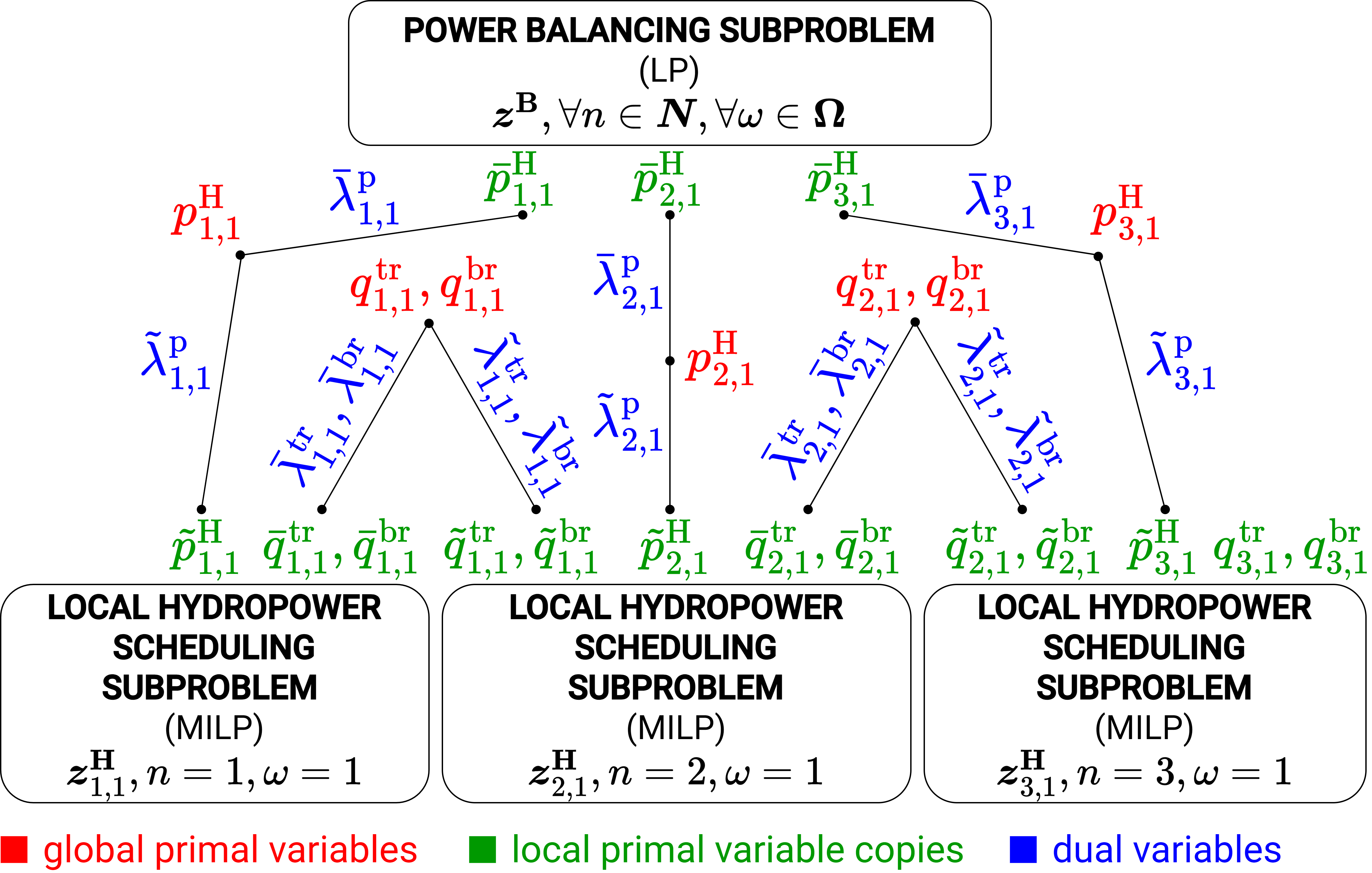}
    \caption{Example of decomposed structure of problem \eqref{eq:cen_optim_problem} when reformulated as the consensus problem \eqref{eq:decomposed_optim_problem}. The edges in the bipartite graph represent the consistency constraint linking global variables with their local copies.}
    \label{fig:decomposed_problem_example}
\end{figure}

Clearly, problem \eqref{eq:decomposed_optim_problem} exhibits a decomposable structure. Subproblem \eqref{eq:decomposed_optim_problem_sub1} pertains to the VPP power balancing and is referred to as the \textit{power balancing subproblem}. Subproblems \eqref{eq:decomposed_optim_problem_sub2} address the hydropower plants individual dynamics and are referred to as the \textit{local hydropower scheduling subproblems}. An example of the resulting decomposed structure is shown in Fig.~\ref{fig:decomposed_problem_example} for $\left|\boldsymbol{N}\right| = 3$ and $\left|\boldsymbol{\Omega}\right| = 1$. The computational advantage of this decomposition is that the original stochastic MILP problem, solved for all $n$ and $\omega$, is now broken down into $\left|\boldsymbol{N} \right| \times \left|\boldsymbol{\Omega} \right|$ deterministic MILP subproblems, and one stochastic LP subproblem, solved for all $n$ and $\omega$.

As illustrated in Fig.~\ref{fig:decomposed_problem_example}, the original global variables of the centralized problem \eqref{eq:cen_optim_problem}, namely $\boldsymbol{q}^{\boldsymbol{\mathrm{tr}}}_{n,\omega}$, $\boldsymbol{q}^{\boldsymbol{\mathrm{br}}}_{n,\omega}$, and $\boldsymbol{p}^{\boldsymbol{\mathrm{H}}}_{n,\omega}$, are duplicated in the proposed decomposition. For each set of decoupled coupling constraints, two distinct sets of subproblems are introduced, each operating on its own set of duplicated variables, which are distinguished by distinct superscripts. For example, the global variable $\boldsymbol{p}^{\boldsymbol{\mathrm{H}}}_{n,\omega}$ is decomposed into local copies $\boldsymbol{\bar{p}}^{\boldsymbol{\mathrm{H}}}_{n,\omega}$ for the power balancing subproblem, and distinct local copies $\boldsymbol{\tilde{p}}^{\boldsymbol{\mathrm{H}}}_{n,\omega}$ for the hydropower scheduling subproblems. The dual variables associated with the consensus constraints \eqref{eq:consensus}, which are enforced for each pair of global and local variable copies, are denoted using the same superscript as their corresponding local variables. The resulting scenario-spatial decomposition enables a fully parallelizable solution of the original centralized problem \eqref{eq:cen_optim_problem} via the distributed optimization algorithm detailed in the following subsection.

Analyzing the participation of the VPP examined in this study in markets other than the day-ahead market is outside the scope of this work. For a comprehensive evaluation of its profitability when simultaneously participating in both day-ahead and ancillary services markets, the reader is referred to \cite{santosuosso2025distributed}. Nonetheless, it is worth noting that the proposed scenario-spatial decomposition not only improves the scalability of the bidding strategy compared to traditional centralized approaches, but also enhances its flexibility and modularity. Specifically, the proposed decomposition approach decouples the local hydropower scheduling from the global power balancing of the VPP, resulting in two distinct sets of subproblems (see Fig.~\ref{fig:decomposed_problem_example}). In multi-market extensions, modifications to the bidding optimization model are typically confined to the global balancing subproblem and may include adjustments to the power balance constraints, updates to the objective function to reflect additional market revenues or penalties, and the inclusion of new constraints on the VPP’s cumulative power generation \cite{santosuosso2025distributed}. The local hydropower scheduling subproblems in Fig.~\ref{fig:decomposed_problem_example} remain unchanged. This modular structure supports both scalability and adaptability to the dynamic evolution of the VPP operations and objectives.

\subsection{Distributed Optimization with Trustworthy Bounds}
The reformulation presented in \eqref{eq:decomposed_optim_problem} enables the application of various distributed optimization techniques. Among others, the consensus ADMM has demonstrated significant effectiveness \cite{Boyd}. Let $k$ denote the current iteration of the distributed algorithm, and $\rho^k$ denote the penalty parameter, which is updated dynamically, as detailed in \cite{Boyd}. To simplify the following formulation, define the matrix $\mathbf{A} \in \mathbb{R}^{6|\boldsymbol{T}| \times 6|\boldsymbol{T}|}$ as a diagonal matrix with a $5|\boldsymbol{T}| \times 5|\boldsymbol{T}|$ identity matrix in the top-left block and the last $|\boldsymbol{T}|$ diagonal elements set to zero. By applying the consensus ADMM algorithm to \eqref{eq:decomposed_optim_problem}, the iterative steps outlined below are derived.

\textbf{Step I.} \textit{Local primal variables update} \eqref{eq:ADMM_stepIa}-\eqref{eq:ADMM_stepIb}:
\begin{equation}
\label{eq:ADMM_stepIa}
\begin{split}
{\boldsymbol{z}^{\boldsymbol{\mathrm{H}}}_{n,\omega}}^{k+1} = & \argmin_{\boldsymbol{z}^{\boldsymbol{\mathrm{H}}}_{n,\omega} \in \boldsymbol{\Gamma}_{n,\omega}}
\Biggl\{ \left(\mathbf{A} {\boldsymbol{\lambda}_{n,\omega}}^k\right)^\top \boldsymbol{\bar{z}}^{\boldsymbol{\mathrm{B},\mathrm{H}}}_{n,\omega}\\
& + \frac{\rho}{2} \left\| \mathbf{A} \left(\boldsymbol{\bar{z}}^{\boldsymbol{\mathrm{B},\mathrm{H}}}_{n,\omega} - {\boldsymbol{z}^{\boldsymbol{\mathrm{B},\mathrm{H}}}_{n,\omega}}^k \right)\right\|^2_2 \Biggl\}, \; \forall n, \forall \omega,
\end{split}
\end{equation}
\begin{equation}
\label{eq:ADMM_stepIb}
\begin{split}
{\boldsymbol{z^{\mathrm{B}}}}^{k+1} = & \argmin_{\boldsymbol{z^{\mathrm{B}}} \in \boldsymbol{\Xi}}
\Biggl\{ J + \sum_{n \in \boldsymbol{N}} \sum_{\omega \in \boldsymbol{\Omega}} \biggl( \left({\boldsymbol{\bar{\lambda}}^{\boldsymbol{\mathrm{p}}}_{n,\omega}}^k\right)^\top \boldsymbol{\bar{p}}^{\boldsymbol{\mathrm{H}}}_{n,\omega}\\
&+ \frac{\rho}{2} \left\|\boldsymbol{\bar{p}}^{\boldsymbol{\mathrm{H}}}_{n,\omega} - {\boldsymbol{p}^{\boldsymbol{\mathrm{H}}}_{n,\omega}}^{k} \right\|^2_2 \biggl) \Biggl\}.
\end{split}
\end{equation}

\textbf{Step II.} \textit{Global primal variables update} in \eqref{eq:ADMM_stepII}:
\begin{subequations}
\label{eq:ADMM_stepII}
\begin{align}
{\boldsymbol{q}^{\boldsymbol{\mathrm{tr}}}_{n,\omega}}^{k+1} & = \frac{1}{2} \left( {\boldsymbol{\bar{q}}^{\boldsymbol{\mathrm{tr}}}_{n,\omega}}^{k+1} + {\boldsymbol{\tilde{q}}^{\boldsymbol{\mathrm{tr}}}_{n,\omega}}^{k+1}\right), \; \forall n, \forall \omega,\\
{\boldsymbol{q}^{\boldsymbol{\mathrm{br}}}_{n,\omega}}^{k+1} &= \frac{1}{2} \left( {\boldsymbol{\bar{q}}^{\boldsymbol{\mathrm{br}}}_{n,\omega}}^{k+1} + {\boldsymbol{\tilde{q}}^{\boldsymbol{\mathrm{br}}}_{n,\omega}}^{k+1}\right), \; \forall n, \forall \omega,\\
{\boldsymbol{p}^{\boldsymbol{\mathrm{H}}}_{n,\omega}}^{k+1} &= \frac{1}{2} \left( {\boldsymbol{\bar{p}}^{\boldsymbol{\mathrm{H}}}_{n,\omega}}^{k+1} + {\boldsymbol{\tilde{p}}^{\boldsymbol{\mathrm{H}}}_{n,\omega}}^{k+1}\right), \; \forall n, \forall \omega.
\end{align}
\end{subequations}

\textbf{Step III.} \textit{Dual variables update} in \eqref{eq:ADMM_stepIII}:
\begin{equation}
\label{eq:ADMM_stepIII}
{\boldsymbol{\lambda}}^{k+1}_{n,\omega} = {\boldsymbol{\lambda}}^{k}_{n,\omega} + \rho \left({\boldsymbol{\bar{z}}^{\boldsymbol{\mathrm{B},\mathrm{H}}}_{n,\omega}}^{k+1} - {\boldsymbol{z}^{\boldsymbol{\mathrm{B},\mathrm{H}}}_{n,\omega}}^{k+1} \right), \; \forall n, \forall \omega.
\end{equation}

The simulation results to be presented in the next section show that applying ADMM to solve \eqref{eq:decomposed_optim_problem} often yields high-quality solutions, highlighting its potential. However, the method’s convergence depends on the convexity of the subproblems, which is not met here due to the presence of binary variables, making ADMM a promising yet heuristic approach for solving \eqref{eq:decomposed_optim_problem}. To make it trustworthy, a method is introduced to derive effective upper and lower bounds on the optimal objective function value of \eqref{eq:cen_optim_problem}, denoted by $J^\star$, based on the distributed solution computed at each iteration of the consensus ADMM.

The initial lower bound, ${J^{\mathrm{LB}}}^0$, is obtained by solving the linear relaxation of \eqref{eq:cen_optim_problem}. This relaxation expands the feasible space, inevitably yielding a new optimum that is less than or equal to $J^\star$. Conversely, the initial upper bound, ${J^{\mathrm{UB}}}^0$, is established by constraining the feasible space, which can be achieved by eliminating the flexibility of CHPP. This is done by setting the upper and lower limits on the forebay water level of each hydropower plant to be identical, thus fixing the binary variables in \eqref{eq:z_fbl_limits}.
Consequently, the binary variables $b^\mathrm{br}$ in \eqref{eq:q_br_safety1} are fixed to 1 for all plants, at every time step and scenario, allowing the constraints \eqref{eq:q_br_safety2} to be neglected.
Determining both ${J^{\mathrm{LB}}}^0$ and ${J^{\mathrm{UB}}}^0$ involves solving two LP problems before executing the distributed algorithm, with minimal computational impact. The proposed method then iteratively updates ${J^{\mathrm{LB}}}^0$ and ${J^{\mathrm{UB}}}^0$ as the ADMM algorithm proceeds, following two distinct strategies.

For the upper bound, each iteration $k$ yields an updated value, ${J^{\mathrm{UB}}}^k$, through a \textit{quasi-projection step}. Let $\boldsymbol{z}^k$ denote the primal variable value obtained via ADMM. Drawing inspiration from the outer approximation strategy \cite{López-Salgado}, $\boldsymbol{z}^k$ is projected onto the LP formed by fixing the binary variables in \eqref{eq:cen_optim_problem} to the values of $\boldsymbol{z}^k$. This projection, denoted by ${\boldsymbol{z}^{\mathrm{prj}}}^k$, remains feasible within \eqref{eq:cen_optim_problem} since it is derived from a projection onto a new feasible space formed by constraining \eqref{eq:cen_optim_problem}—specifically, fixing the binary variables. Thus, $J^\star \leq {J^{\mathrm{UB}}}^k \coloneqq J\left({\boldsymbol{z}^{\mathrm{prj}}}^k\right)$.

Proposition~\ref{proposition} describes the process of using dual variables associated with the consensus constraints \eqref{eq:decomposed_optim_problem} for the iterative updating of the lower bound ${J^{\mathrm{LB}}}^k$. It extends the theoretical findings from \cite{Gade}—established for consensus ADMM in scenario decomposition—to the case where the problem is decomposed across both spatial and scenario dimensions.
Notably, although stated for $J$ in \eqref{eq:cen_obj_fun}, the result of Proposition~\ref{proposition} readily extends to any alternative convex objective function.

\begin{proposition}
\label{proposition}
Let $\boldsymbol{\lambda} \coloneqq \left\{\boldsymbol{\lambda}_{n,\omega}\right\}_{n \in \boldsymbol{N}, \omega \in \boldsymbol{\Omega}}$ satisfy the following conditions (component-wise):
\begin{align}
\sum_{n \in \boldsymbol{N}} \sum_{\omega \in \boldsymbol{\Omega}} P_\omega \left(\boldsymbol{\bar{\lambda}}^{\boldsymbol{\mathrm{p}}}_{n,\omega} + \boldsymbol{\tilde{\lambda}}^{\boldsymbol{\mathrm{p}}}_{n,\omega}\right) = 0, \label{eq:prop_assumption1}\\
\sum_{n \in \boldsymbol{N}} \sum_{\omega \in \boldsymbol{\Omega}} P_\omega \left(\boldsymbol{\bar{\lambda}}^{\boldsymbol{\mathrm{tr}}}_{n,\omega} + \boldsymbol{\tilde{\lambda}}^{\boldsymbol{\mathrm{tr}}}_{n,\omega}\right) = 0, \label{eq:prop_assumption2}\\
\sum_{n \in \boldsymbol{N}} \sum_{\omega \in \boldsymbol{\Omega}} P_\omega \left(\boldsymbol{\bar{\lambda}}^{\boldsymbol{\mathrm{br}}}_{n,\omega} + \boldsymbol{\tilde{\lambda}}^{\boldsymbol{\mathrm{br}}}_{n,\omega}\right) = 0. \label{eq:prop_assumption3}
\end{align}
Let
\begin{equation}
D^{\mathrm{B}}\left(\boldsymbol{\lambda}\right) \coloneqq \min_{\boldsymbol{z}^{\boldsymbol{\mathrm{B}}} \in \boldsymbol{\Xi}} \Biggl\{ J + \sum_{n \in \boldsymbol{N}} \sum_{\omega \in \boldsymbol{\Omega}} P_\omega \left({\boldsymbol{\bar{\lambda}}^{\boldsymbol{\mathrm{p}}}_{n,\omega}}\right)^\top \boldsymbol{p}^{\boldsymbol{\mathrm{H}}}_{n,\omega} \Biggl\},
\end{equation}
and
\begin{multline}
D^{\mathrm{H}}_{n,\omega}\left(\boldsymbol{\lambda}_{n,\omega}\right) \coloneqq \min_{\boldsymbol{z}^{\boldsymbol{\mathrm{H}}}_{n,\omega} \in \boldsymbol{\Gamma}_{n,\omega}}
\left\{ \left(\mathbf{A} {\boldsymbol{\lambda}_{n,\omega}}\right)^\top \boldsymbol{z}^{\boldsymbol{\mathrm{B},\mathrm{H}}}_{n,\omega}
\right\},\\ \forall n, \forall \omega.
\end{multline}
Then
\begin{equation}
D\left(\boldsymbol{\lambda}\right) \coloneqq D^{\mathrm{B}}(\boldsymbol{\lambda}) + \sum_{n \in \boldsymbol{N}} \sum_{\omega \in \boldsymbol{\Omega}} P_{\omega} D^{\mathrm{H}}_{n,\omega}\left(\boldsymbol{\lambda}_{n, \omega}\right) \leq J^\star.
\end{equation}
\end{proposition}
\begin{proof}
Let $\left\{\boldsymbol{b}^{\boldsymbol{\mathrm{br}}^\star}, \boldsymbol{b}^{\boldsymbol{\mathrm{oc}}^\star} \boldsymbol{e}^\star, \boldsymbol{\delta^{\mathrm{E}}}^\star, 
\boldsymbol{p}^{\boldsymbol{\mathrm{H}}^\star}, \boldsymbol{q}^{\boldsymbol{\mathrm{tr}}^\star}, \boldsymbol{q}^{\boldsymbol{\mathrm{br}}^\star}, \boldsymbol{z}^{\boldsymbol{\mathrm{fbl}}^\star}\right\}$ be an optimal solution of the original problem \eqref{eq:cen_optim_problem}. Feasibility implies that this optimal solution is also feasible within the relaxed feasible spaces of the subproblems \eqref{eq:decomposed_optim_problem_sub1} and \eqref{eq:decomposed_optim_problem_sub2}, i.e., it is feasible in $\boldsymbol{\Xi}$ and $\boldsymbol{\Gamma}_{n,\omega}, \forall n, \forall \omega$, respectively. Thus:
\begin{equation}
D^{\mathrm{B}}(\boldsymbol{\lambda}) \leq J(\boldsymbol{e}^\star, \boldsymbol{\delta^{\mathrm{E}}}^\star) + \sum_{n \in \boldsymbol{N}} \sum_{\omega \in \boldsymbol{\Omega}} P_\omega \left({\boldsymbol{\bar{\lambda}}^{\boldsymbol{\mathrm{p}}}_{n,\omega}}\right)^\top \boldsymbol{p}^{\boldsymbol{\mathrm{H}}^\star}_{n,\omega},
\end{equation}
and
\begin{equation}
\begin{split}
D^{\mathrm{H}}(\boldsymbol{\lambda}_{n,\omega}) \leq &
\left({\boldsymbol{\tilde{\lambda}}^{\boldsymbol{\mathrm{p}}}_{n,\omega}}\right)^\top \boldsymbol{p}^{{\boldsymbol{\mathrm{H}}}^\star}_{n,\omega}
+ \left(\boldsymbol{\tilde{\lambda}}_{n-1,\omega}^{\boldsymbol{\mathrm{tr}}}\right)^\top \boldsymbol{q}_{n-1,\omega}^{\boldsymbol{\mathrm{tr}}^\star}\\
& + \left(\boldsymbol{\tilde{\lambda}}_{n-1,\omega}^{\boldsymbol{\mathrm{br}}}\right)^\top \boldsymbol{q}_{n-1,\omega}^{\boldsymbol{\mathrm{br}}^\star}
+ \left(\boldsymbol{\bar{\lambda}}_{n,\omega}^{\boldsymbol{\mathrm{tr}}}\right)^\top \boldsymbol{q}^{\boldsymbol{\mathrm{tr}}^\star}_{n,\omega}\\
& + \left(\boldsymbol{\bar{\lambda}}_{n,\omega}^{\boldsymbol{\mathrm{br}}}\right)^\top \boldsymbol{q}^{\boldsymbol{\mathrm{br}}^\star}_{n,\omega} \; , \; \forall n, \forall \omega.
\end{split}
\end{equation}
Then
\begin{equation}
\begin{split}
D(\boldsymbol{\lambda}) \leq & \sum_{n \in \boldsymbol{N}} \sum_{\omega \in \boldsymbol{\Omega}} P_\omega \left(\boldsymbol{\bar{\lambda}}_{n,\omega}^{\boldsymbol{\mathrm{p}}} + \boldsymbol{\tilde{\lambda}}_{n,\omega}^{\boldsymbol{\mathrm{p}}}\right)^\top \boldsymbol{p}_{n,\omega}^{\boldsymbol{\mathrm{H}}^\star}\\
& + \sum_{n \in \boldsymbol{N}} \sum_{\omega \in \boldsymbol{\Omega}} P_\omega \left(\boldsymbol{\bar{\lambda}}_{n,\omega}^{\boldsymbol{\mathrm{tr}}} + \boldsymbol{\tilde{\lambda}}_{n,\omega}^{\boldsymbol{\mathrm{tr}}}\right)^\top \boldsymbol{q}^{\boldsymbol{\mathrm{tr}}^\star}_{n,\omega} \\
& +
\sum_{n \in \boldsymbol{N}} \sum_{\omega \in \boldsymbol{\Omega}} P_\omega \left(\boldsymbol{\bar{\lambda}}_{n,\omega}^{\boldsymbol{\mathrm{br}}} + \boldsymbol{\tilde{\lambda}}_{n,\omega}^{\boldsymbol{\mathrm{br}}}\right)^\top \boldsymbol{q}_{n,\omega}^{\boldsymbol{\mathrm{br}}^\star}\\
& + J(\boldsymbol{e}^\star, \boldsymbol{\delta^{\mathrm{E}}}^\star)\\
= & J(\boldsymbol{e}^\star, \boldsymbol{\delta^{\mathrm{E}}}^\star) = J^\star.
\end{split}
\end{equation}
The second last equality follows from \eqref{eq:prop_assumption1}-\eqref{eq:prop_assumption3}, which are inherently satisfied by the ADMM, as shown in \cite{Boyd}.
\end{proof}

The proposed algorithm is referred to as the \textbf{Consensus ADMM with Bounds} (\textbf{CADMMB}) and is summarized in Fig.~\ref{fig:Consensus_ADMM_bounded}. A tolerance parameter $\epsilon^{\mathrm{UB}}$ is used to guide whether the upper bound update step should be performed at each iteration. The termination criterion is given in \eqref{eq:CADMMB_termination_criterion}:
\begin{equation}
\label{eq:CADMMB_termination_criterion}
{J^{\mathrm{gap}}}^k = 100 \frac{\left|{J^{\mathrm{UB}}}^k - {J^{\mathrm{LB}}}^k\right|}{\left|{J^{\mathrm{UB}}}^k\right|} \leq \epsilon^{\mathrm{gap}} \text{ or } k \geq \overline{K},
\end{equation}
where $\epsilon^{\mathrm{gap}}$ and $\overline{K}$ are parameters of the algorithm.

A key advantage of this approach, in contrast to heuristic applications of distributed methods, lies in the availability of \textit{trustworthy} bounds on the global optimum: the solution provided by the CADMMB is $\epsilon^{\mathrm{gap}}$-suboptimal, meaning its objective value is at most $\epsilon^{\mathrm{gap}}$ (e.g., 0.01\%) worse than the value achieved by the optimal solution. In other words, the CADMMB offers a \textit{performance guarantee} by providing solutions that operators can confidently evaluate and trust.

\begin{figure}[!t]
\centering
\includegraphics[scale=.065]{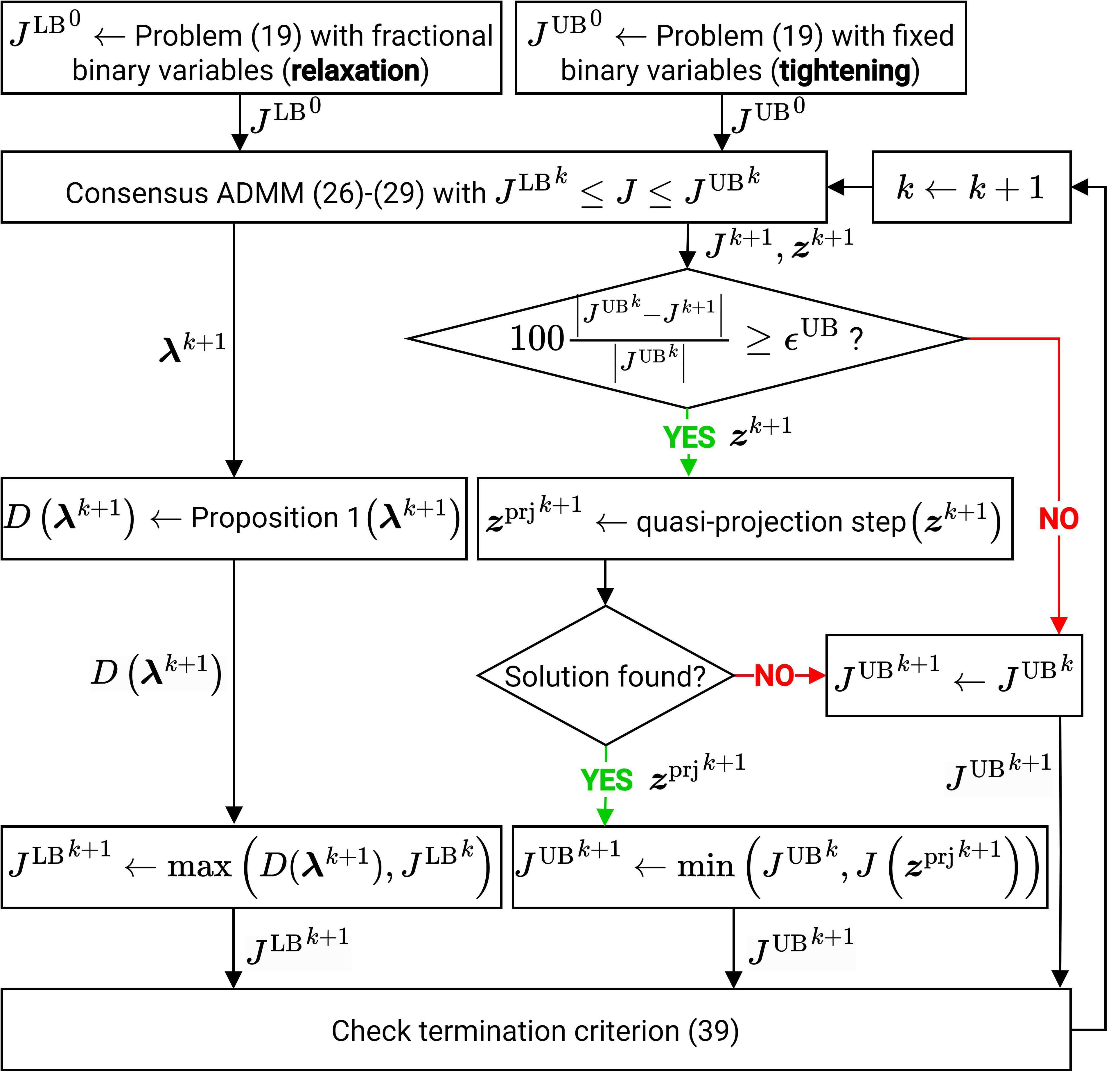}
    \caption{Illustration of the proposed CADMMB algorithm.}
    \label{fig:Consensus_ADMM_bounded}
\end{figure}

\begin{table}[!t]
\caption{Characterization of the uncertainties.}
\label{tab:uncertainty_characterization}
\centering
\begin{tabular}{|c|c|c|c|}
\hline
\multirow{ 2}{*}{\textbf{Uncertainty}} & \multicolumn{3}{c|}{\textbf{Average hourly values (standard deviation)}} \\
\cline{2-4}
& \textbf{February} & \textbf{March} & \textbf{April}\\
\hline
$\hat{\pi}^{\mathrm{E}}$ (\euro/MWh) & 51.1 (12.4) & 35.4 (9.8) & 34.7 (8.8)\\
$\hat{\pi}^{\mathrm{E}\uparrow}$ (\euro/MWh)  & 58.1 (13.9) & 38.6 (10.7) & 38.6 (9.7)\\
$\hat{\pi}^{\mathrm{E}\downarrow}$ (\euro/MWh)  & 49.1 (12.8) & 32.3 (8.4) & 28.5 (8.3)\\
$\hat{Q}^{\mathrm{ext}}$ (m$^3$/s)  & 1367.8 (456.1) & 1932.2 (549.3) & 798.68 (219.0)\\
$\hat{P}^\mathrm{S}$ (MWh) & 6.2 (10.7) & 9.6 (14.1) & 13.5 (17.4)\\
$\hat{P}^\mathrm{W}$ (MWh) & 91.5 (61.9) & 84.6 (57.1) & 52.8 (29.5)\\
\hline
\end{tabular}
\end{table}

\begin{table}[!t]
\caption{Characterization of the hydropower cascade.}
\label{tab:hydro_cascade_characterization}
\centering
\begin{tabular}{|c|c|c|c|}
\hline
\multirow{2}{*}{\textbf{Asset name}} & \textbf{Average annual} & \textbf{Head} & \textbf{Capacity}\\
& \textbf{energy output (GWh)} & \textbf{(m)} & \textbf{(MW)}\\
\hline
Pierre-Bénite & 528 & 9.0 & 80\\
Vaugris & 332 & 6.7 & 72\\
Péage de Roussillon & 885 & 12.2 & 160\\
Saint Vallier & 668 & 11.5 & 120\\
Bourg lès Valence & 1082 & 11.7 & 180\\
Beauchastel & 1211 & 11.8 & 192\\
Baix le Logis Neuf & 1177 & 11.7 & 192\\
Montélimar & 1575 & 16.5 & 275\\
Donzère-Mondragon & 2032 & 22.5 & 349\\
Caderousse & 843 & 8.6 & 187\\
Avignon & 857 & 9.5 & 186\\
Vallabrègues & 1269 & 11.3 & 210\\
\hline
\end{tabular}
\end{table}

\section{Simulation Results and Discussion}
\label{sec:results}
This section presents simulation results over three months (February-April 2017), covering diverse inflow conditions. The uncertainties are characterized in Table~\ref{tab:uncertainty_characterization}.
Compagnie Nationale du Rhône's analysis shows that the uncertainties in vRES, inflows, and prices are minimally correlated and can be treated as independent \cite{santosuosso2025distributed}.
Based on this finding, distinct sets of scenarios are generated for each uncertainty source, following the methodology utilized in the real-world operations of Compagnie Nationale du Rhône \cite{santosuosso2025distributed}.
A scenario reduction process is then applied, wherein a random subset of the original set of scenarios is selected, with equal probability assigned to each scenario.
To ensure accurate uncertainty representation, the selected scenarios must adequately capture the dispersion of the original set. Consequently, an out-of-sample validation procedure is employed to determine the minimum number of scenarios necessary to adequately represent the underlying uncertainty \cite{ARRIGO2022304}.
It is important to note that various methods exist for both scenario generation and reduction. A comprehensive discussion of these methods is outside the scope of this study but can be found in the existing literature dedicated to this topic \cite{LI2022107722}.
However, it is worth emphasizing that no scenario generation method can generally guarantee a perfect a priori characterization of the uncertainty. 
Therefore, the significance of the analysis presented here lies in its out-of-sample validation results, which assess the reliability of the selected scenarios in accurately representing the underlying uncertainty, irrespective of the methods employed for scenario generation and reduction.

The hydropower cascade configuration is shown in Fig.~\ref{fig:vpp} and detailed in Table~\ref{tab:hydro_cascade_characterization}. Simulations use historical day-ahead market prices in France, with the bidding problem solved iteratively each day within a 4-hour optimization window.
A piecewise linearization of the operational curves in Fig.~\ref{fig:operational_curve} is employed using 40 equally spaced segments, i.e., $|\boldsymbol{S}^{\boldsymbol{\mathrm{oc}}}_{n}| = 40$ for all $n \in \boldsymbol{N}$.
In the CADMMB, both $\epsilon^{\mathrm{gap}}$ and $\epsilon^{\mathrm{UB}}$ are set to 0.01\%, while $\overline{K} = 5000$. The algorithm runs on an Intel i7 processor using Gurobi 9.5.0 as the solver.
The proposed CADMMB algorithm, illustrated in Fig.~\ref{fig:Consensus_ADMM_bounded}, is implemented using the Python’s multiprocessing module \cite{multiprocessing}, enabling fully parallel execution of the distributed subproblems. Each subproblem is assigned to an independent process with separate memory allocation. This design ensures true parallelism and reflects realistic operational conditions, as the derived subproblems correspond to distinct units in the hydropower cascade and are expected to run on separate computational devices with isolated memory.

\begin{figure}[!t]
\centering
\includegraphics[scale=.44]{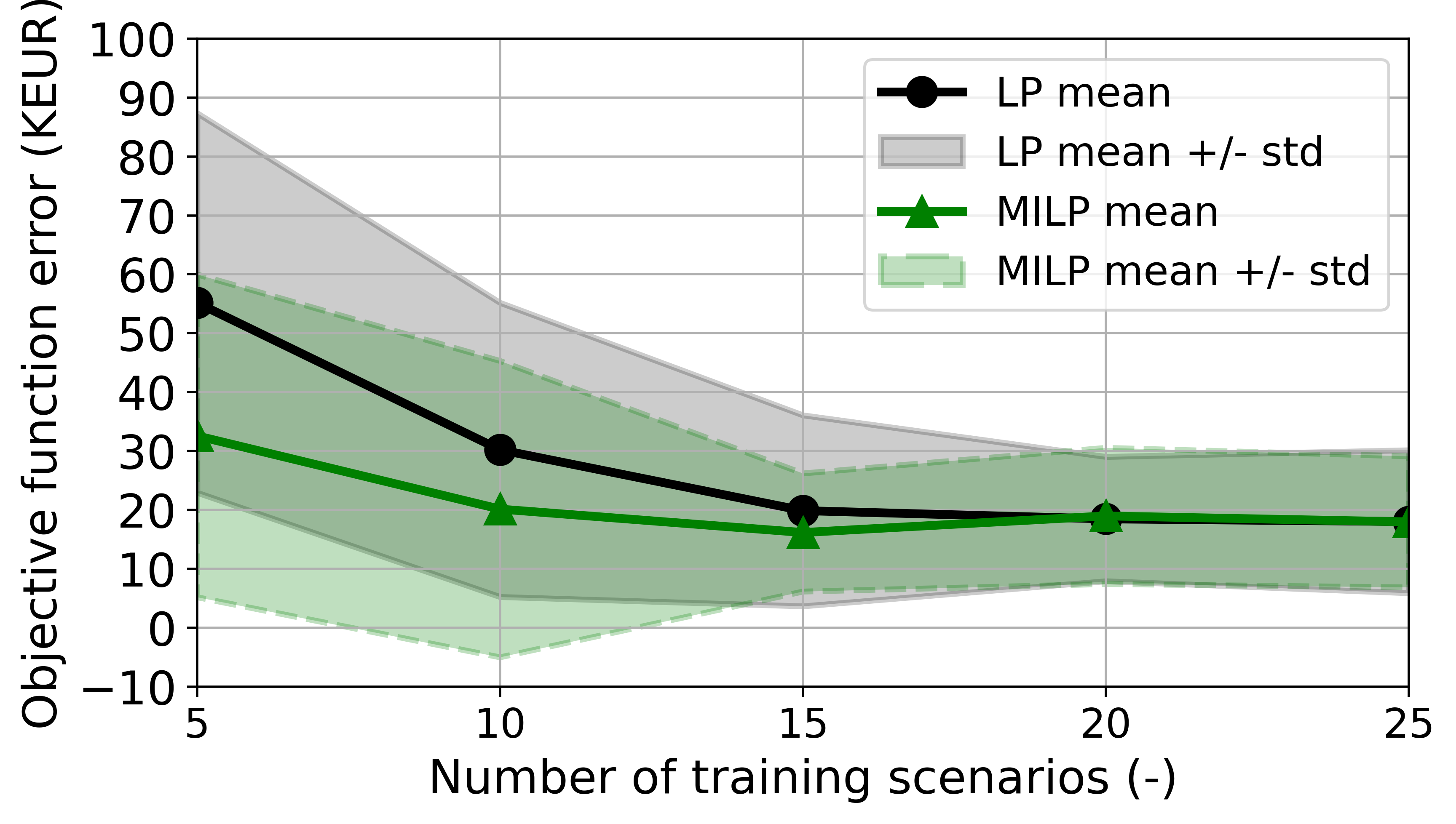}
    \caption{Out-of-sample validation results considering 2000 testing scenarios.}
    \label{fig:out_of_sample_validation}
\end{figure}

\begin{figure}[!t]
\centering
\includegraphics[scale=.49]{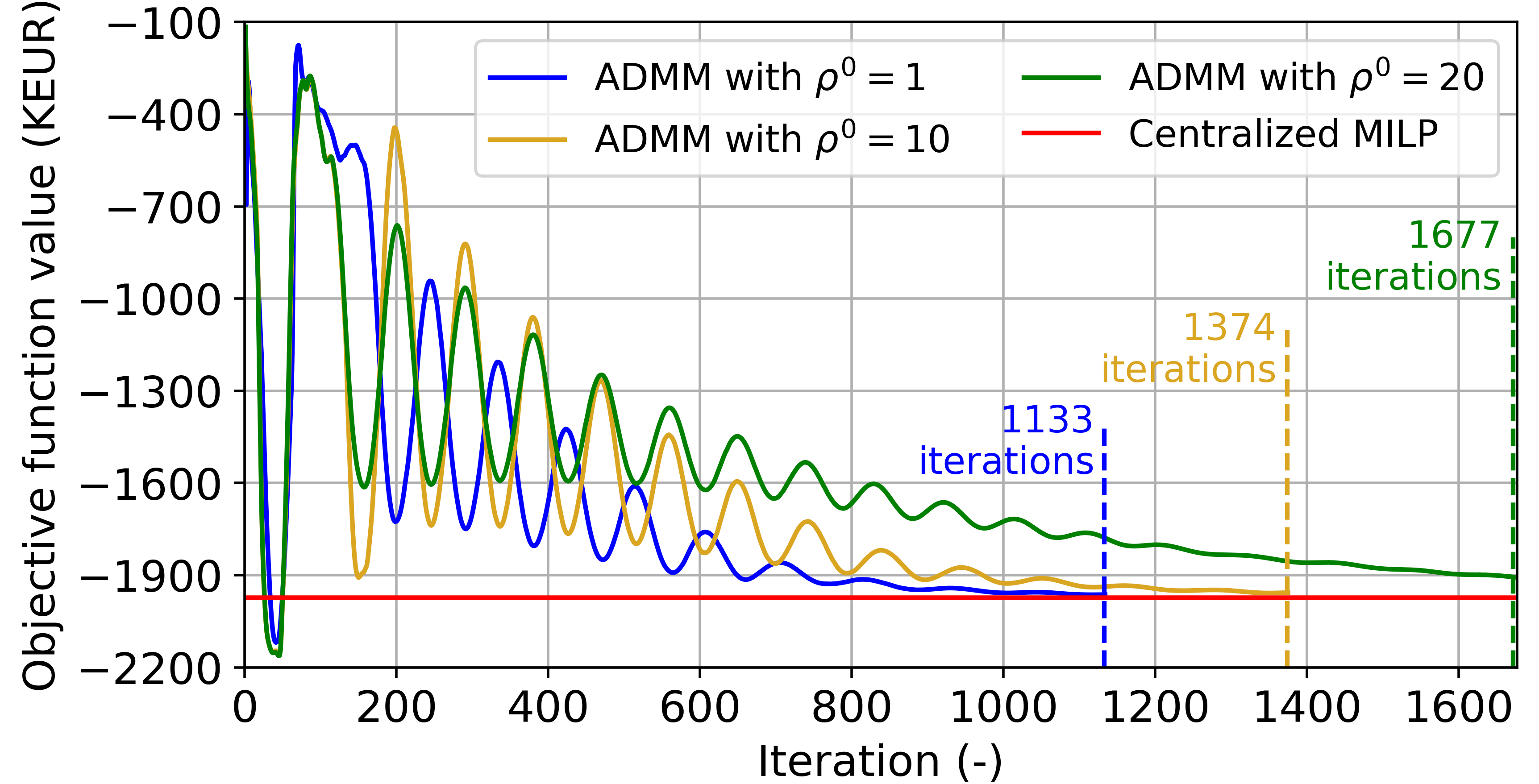}
    \caption{Example of convergence of the standard consensus ADMM algorithm for various initial penalty parameter values, when considering 5 scenarios.}
    \label{fig:ADMM_small_example}
\end{figure}

\begin{figure}[!t]
\centering
\includegraphics[scale=.49]{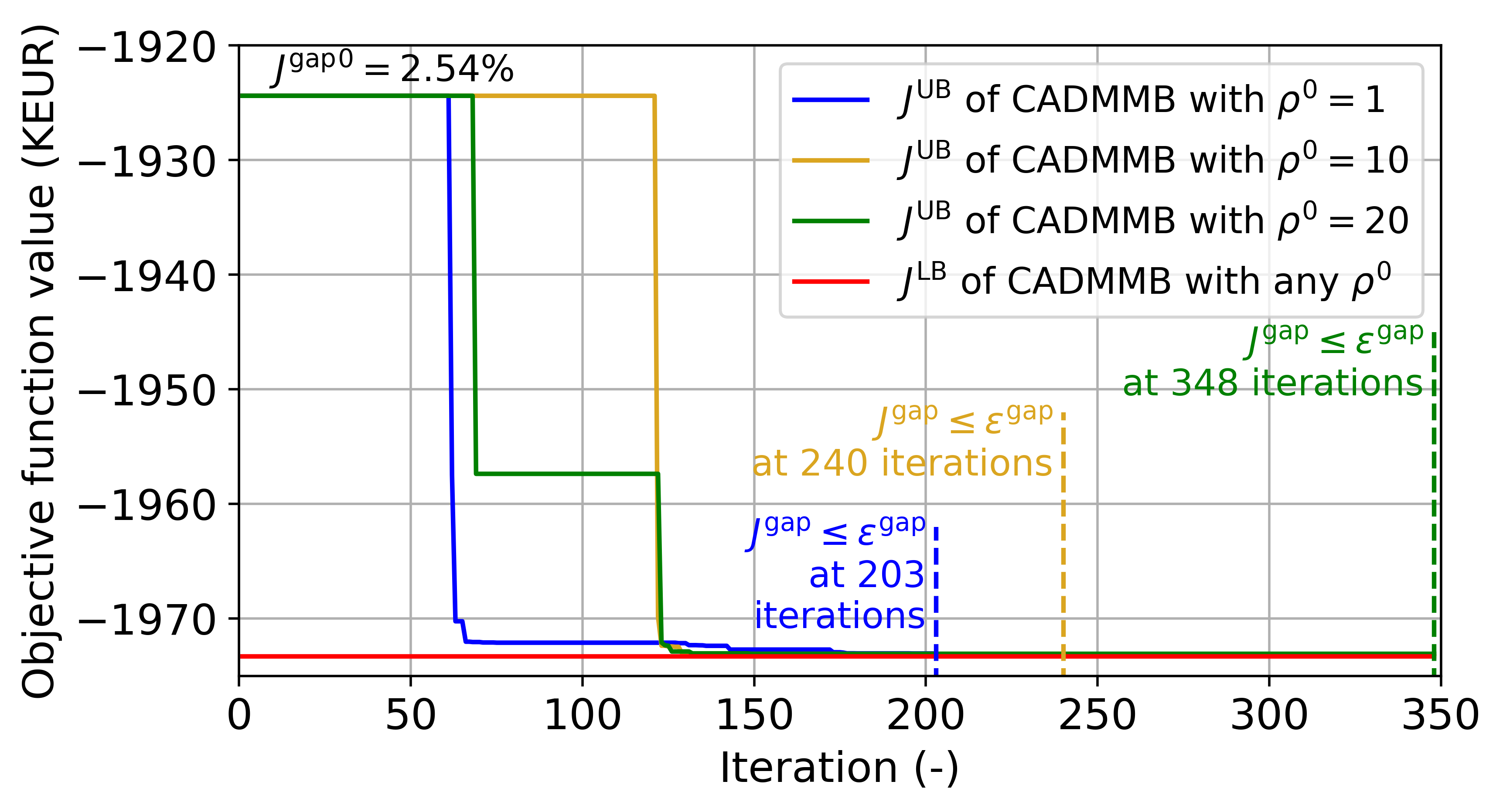}
    \caption{Example of convergence of the proposed CADMMB algorithm for various initial penalty parameter values, when considering 5 scenarios.}
    \label{fig:CADMMB_small_example}
\end{figure}

\subsection{Out-of-Sample Validation}
This subsection reports the out-of-sample validation results for the proposed stochastic bidding strategy in \eqref{eq:cen_optim_problem}. 
This standard procedure assesses the performance of the stochastic optimization tool \cite{ARRIGO2022304}. It involves solving problem \eqref{eq:cen_optim_problem} using a set of training scenarios, then evaluating the objective function \eqref{eq:cen_obj_fun} by applying the optimal decisions from the training to a new testing set, consisting of 2000 scenarios in this case. The difference between the training and testing objective function values is called the \textit{objective function error}. 

Fig.~\ref{fig:out_of_sample_validation} presents the observed out-of-sample validation results. For each selected number of training scenarios (5, 10, 15, 20, and 25, as indicated in the figure), 100 independent rounds of random sampling are conducted on the original set of generated scenarios. In each round, a distinct set of random samples is selected to train the proposed stochastic optimization tool, which then generates a corresponding training objective function value. The average daily difference between the training and testing objective function values is computed across all 100 rounds. This value is reported in the figure to assess the average out-of-sample performance of the proposed stochastic optimization tool for each selected number of training scenarios, accounting for the variability in outcomes arising from the random selection of training scenarios. The rationale for conducting multiple independent rounds of random scenario sampling for each specified number of training scenarios is to improve the robustness of the out-of-sample validation process.
Fig.~\ref{fig:out_of_sample_validation} reports both the observed mean and standard deviation (std) values.

Two distinct sets of out-of-sample validation results are presented in the figure. The first is obtained using the original formulation of the bidding strategy in \eqref{eq:cen_optim_problem}, modeled as a MILP. The second is derived from an LP relaxation of \eqref{eq:cen_optim_problem}, in which the integer variables are relaxed to take continuous values within the range $[0,1]$.
The results indicate that 20 scenarios are sufficient to minimize the error for both the MILP problem \eqref{eq:cen_optim_problem} and its LP relaxation. Using more than 20 scenarios raises computational complexity without significantly improving accuracy.

\subsection{Stylized Illustrative Example}
Before presenting the full case study, results from a stylized example with 5 scenarios are presented. Figs. \ref{fig:ADMM_small_example} and \ref{fig:CADMMB_small_example} show the convergence rates of the consensus ADMM and the proposed CADMMB, respectively, under different initial values of the penalty parameter $\rho$. The results reveal that the standard ADMM converges in less than 1140 iterations at best but exhibits undesirable oscillatory behavior. In contrast, Fig.~\ref{fig:CADMMB_small_example} shows that the CADMMB not only eliminates the oscillations but also achieves bounds convergence with 99.99\% accuracy in 203 iterations at best. The initial lower bound guess is already close to the global optimum and remains stable throughout the iterations. The initial gap between bounds is relatively small (2.54\%), highlighting the benefits of warm starting the algorithm with initial bounds, as shown in Fig.~\ref{fig:Consensus_ADMM_bounded}. While consensus ADMM shows a satisfactory performance, the actual accuracy of its solutions remains indeterminate without explicitly computing the global optimum through centralized optimization. In contrast, our trustworthy CADMMB not only provides bounds that quantify the ``distance" from the global optimum but also enables precise control over the desired accuracy by tuning the $\epsilon^{\mathrm{gap}}$ parameter.

\begin{figure}[!t]
\centering
\includegraphics[scale=.5]{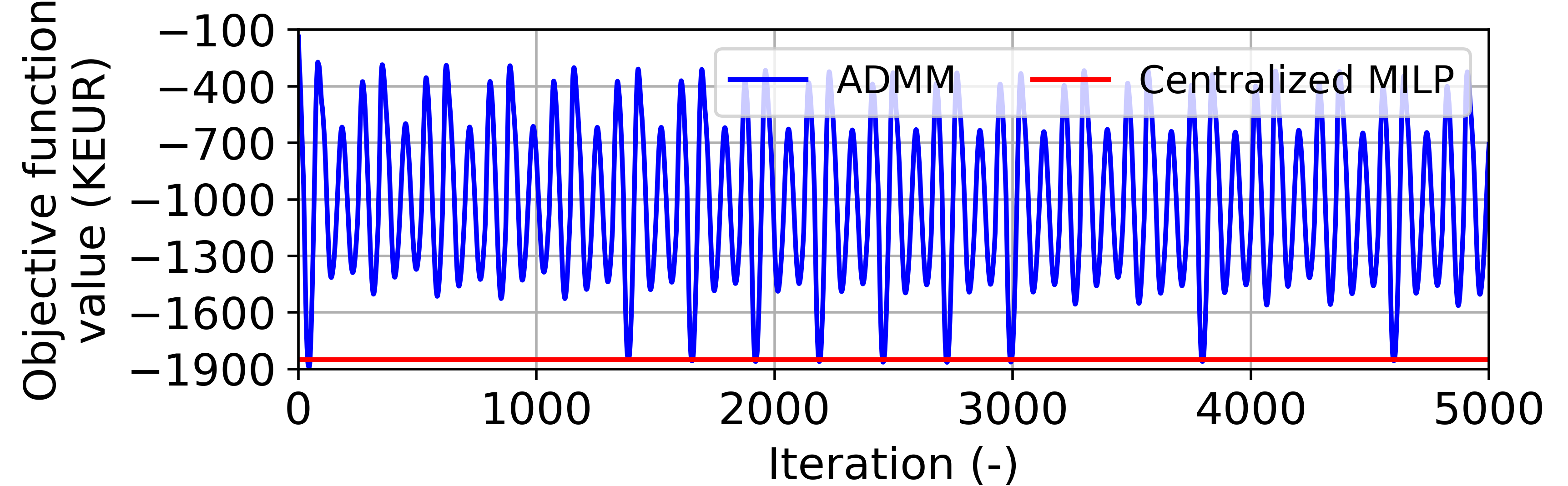}
    \caption{Example of ADMM performance when considering 20 scenarios.}
    \label{fig:ADMM_full_case}
\end{figure}

\begin{figure}[!t]
\centering
\includegraphics[scale=.49]{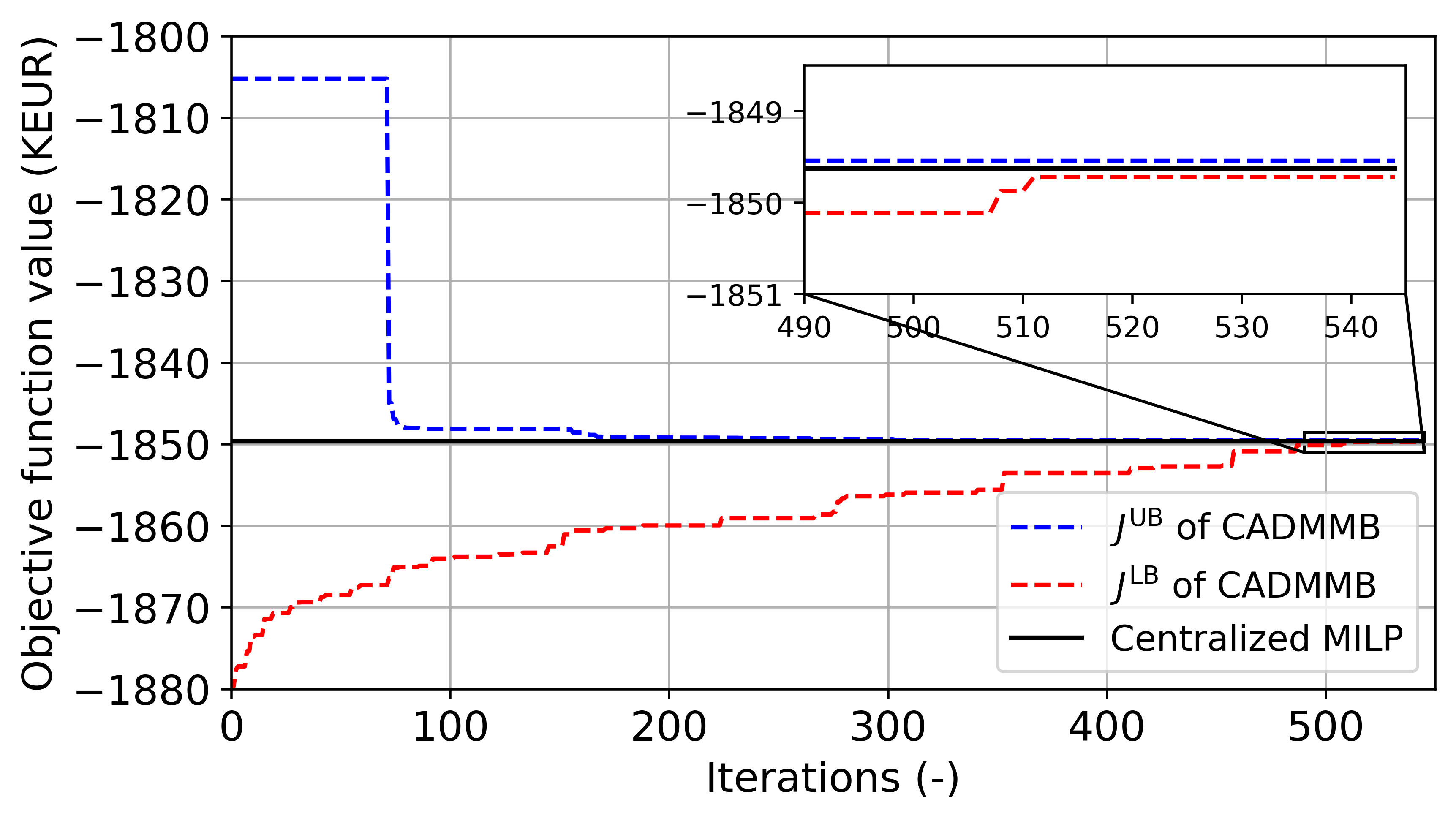}
    \caption{Example of CADMMB performance when considering 20 scenarios.}
    \label{fig:CADMMB_full_case}
\end{figure}

\begin{figure}[!t]
\centering
\includegraphics[scale=.49]{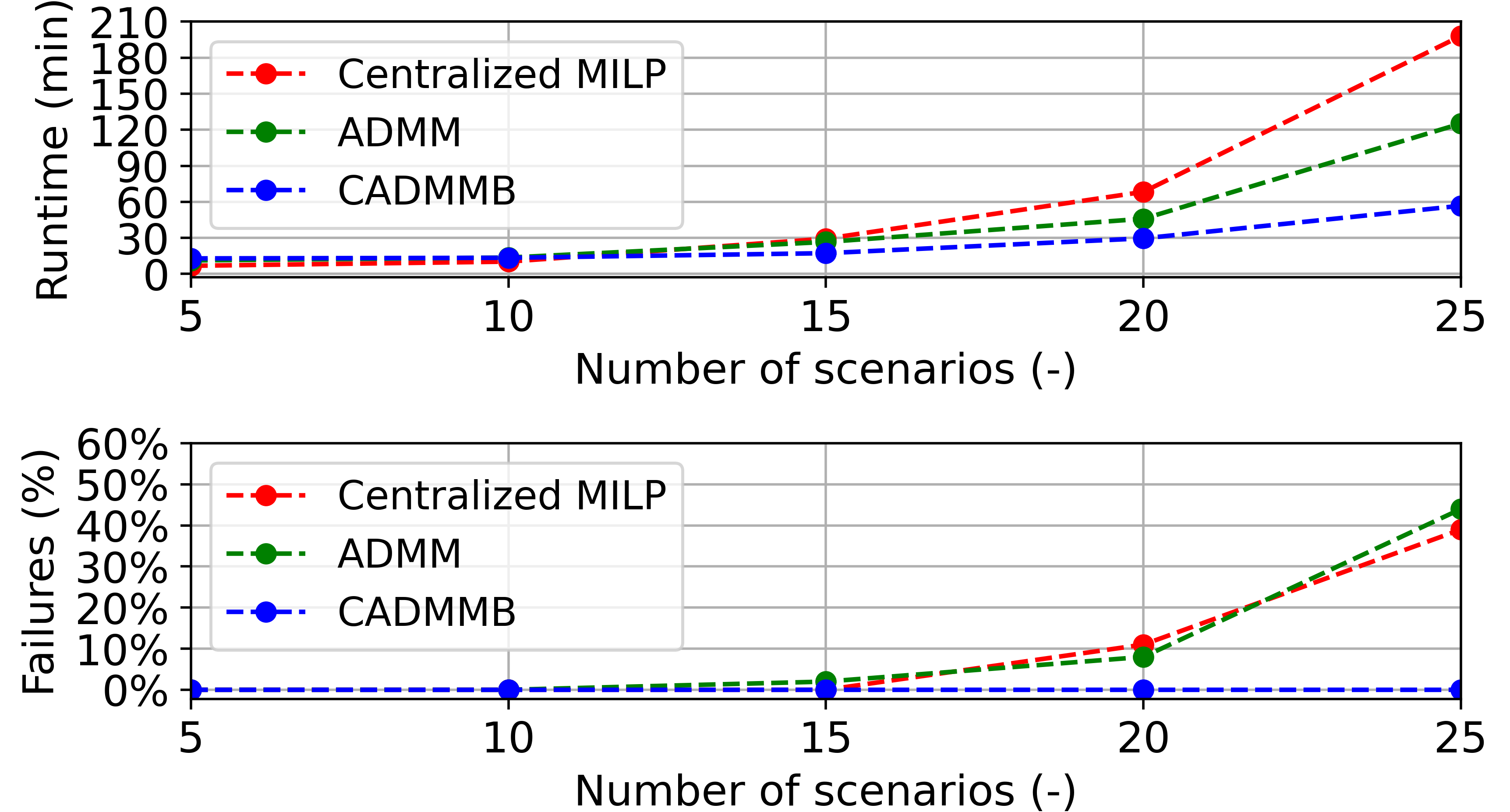}
    \caption{Comparison between centralized MILP, consensus ADMM, and CADMMB based on average runtime per scheduling period and failure rates (\% of total simulation days), as a function of the number of scenarios.}
    \label{fig:runtime_comparison_paper_plot}
\end{figure}

\begin{figure}[!t]
\centering
\includegraphics[scale=.51]{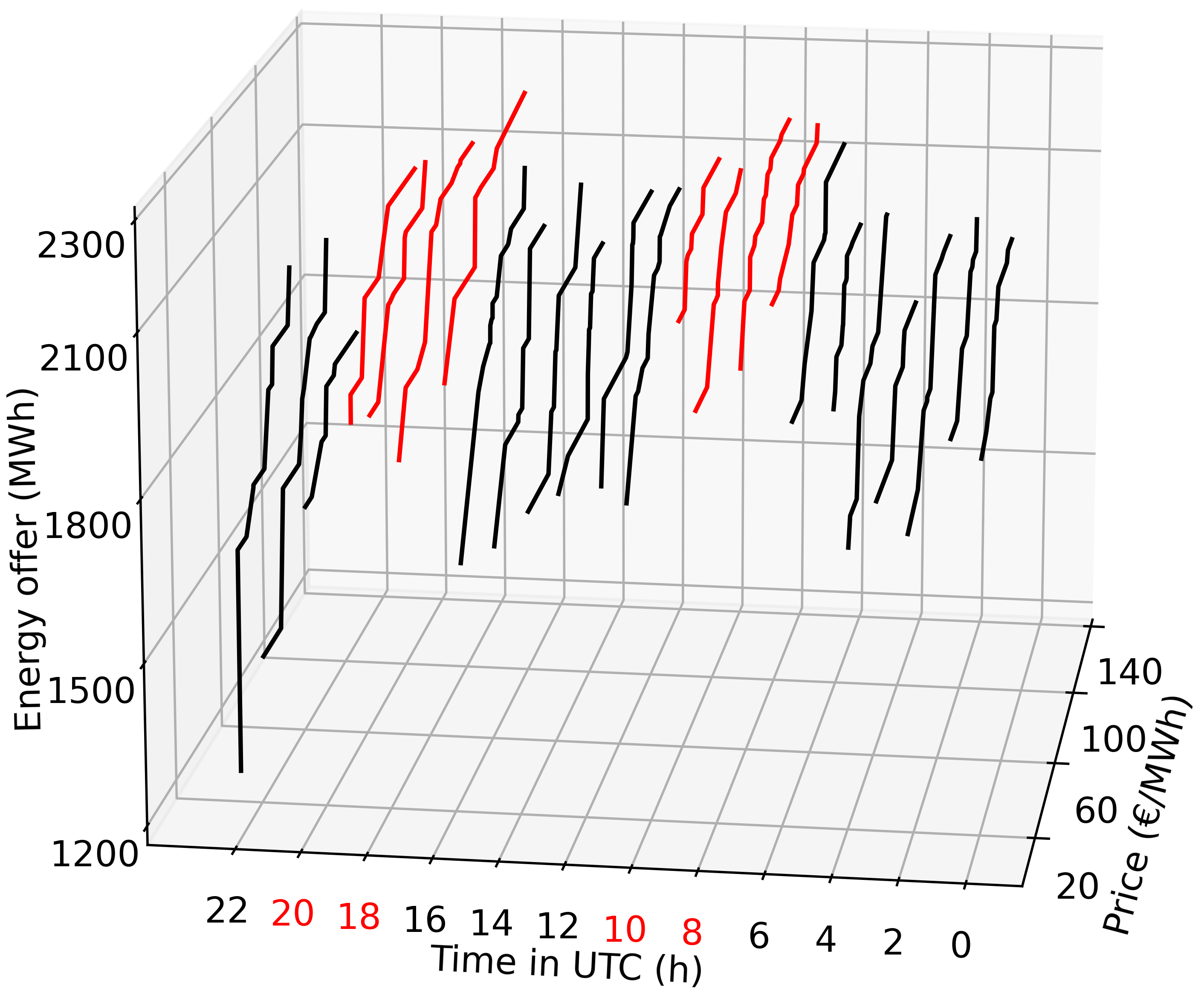}
    \caption{Example of price-quantity bidding curves derived via CADMMB. Peak price periods are highlighted in red.}
    \label{fig:price_quantity_curves_example}
\end{figure}

\begin{table}[!t]
\caption{Comparison between ex-post profit with fixed hourly bids and price-quantity bidding curves. In brackets, the relative difference with respect to fixed hourly bids.}
\label{tab:price_quantity_vs_fixed_bids}
\centering
\begin{tabular}{|c| c c|}
\hline
\multirow{ 2}{*}{\textbf{Num. scenarios}} & \multicolumn{2}{c|}{\textbf{Profit per MWh of energy produced (\euro)}} \\
\cline{2-3}
& \textbf{Fixed bids} & \textbf{Price-quantity bidding curves}\\
\hline
5 & 33.63 & 34.38 ($\boldsymbol{+ 2.2\%}$)\\
10 & 35.74 & 36.80 ($\boldsymbol{+ 3.0\%}$)\\
15 & 37.82 & 39.22 ($\boldsymbol{+ 3.7\%}$)\\
20 & 39.45 & 41.22 ($\boldsymbol{+ 4.5\%}$)\\
25 & 39.89 & 41.36 ($\boldsymbol{+ 3.7\%}$)\\
\hline
\end{tabular}
\end{table}

\subsection{Full Scale Case Study}
Figs. \ref{fig:ADMM_full_case} and \ref{fig:CADMMB_full_case} illustrate the performance of the standard consensus ADMM and the proposed CADMMB, respectively, when applied to the full scale case study involving 20 scenarios. As shown in Fig.~\ref{fig:ADMM_full_case}, the ADMM fails to address the complexity of the problem. In contrast, our CADMMB exhibits a robust performance, efficiently converging with 99.99\% accuracy in about 550 iterations. The initial lower bound estimate in this case is significantly distant from the optimum, necessitating iterative updates to refine $J^{\mathrm{LB}}$. Notably, the evolution of $J^{\mathrm{LB}}$ and $J^{\mathrm{UB}}$ throughout the iterations differs markedly. While $J^{\mathrm{LB}}$ is continuously updated based on the progressive evolution of the dual variables, $J^{\mathrm{UB}}$ is determined through a projection step into a new feasible space, which results in the large jump observed in Fig.~\ref{fig:CADMMB_full_case}.

Fig.~\ref{fig:runtime_comparison_paper_plot} generalizes these findings by showing the average runtime for solving the daily bidding problems and the number of failures—defined as the number of days for which the algorithm fails to provide a solution within the 4-hour optimization window—for different algorithms, as a function of the number of scenarios. In the reference case (20 scenarios), CADMMB reduces the average runtime by approximately 35\% compared to consensus ADMM and by 57\% compared to a traditional centralized optimization. Moreover, while the number of intractable instances explodes for both the standard ADMM and centralized algorithms as the number of scenarios grows, our trustworthy CADMMB consistently provides a distributed solution with over 99.99\% accuracy ($\epsilon^{\mathrm{gap}} = 0.01\%$).

Fig.~\ref{fig:price_quantity_curves_example} illustrates an example of the price-quantity bidding curves generated by our CADMMB. Despite the provided parallelization of the problem shown in Fig.~\ref{fig:decomposed_problem_example}, the CADMMB algorithm effectively coordinates both vRES and hydropower resources to optimize energy offers during daily peak price periods. These findings are generalized in Table~\ref{tab:price_quantity_vs_fixed_bids}, where it is reported that employing our CADMMB with price-quantity bidding curves enhances the average ex-post market profit of the VPP by up to 4.5\% compared to using a traditional centralized solution with fixed hourly bids.

Finally, Fig.~\ref{fig:hydro_vres_example} illustrates a typical daily operation of the last hydropower asset in the cascade using CADMMB, reflecting the overall behavior of the hydropower system. The figure presents three distinct inflow conditions: energetic, low-flow, and flood periods. Due to constraints imposed by hydropower operational curves (e.g., Fig.~\ref{fig:operational_curve}), the hydropower plant can fully respond to vRES energy drops by discharging the reservoir only during energetic periods when sufficient storage capacity is available. During low-flow periods, CADMMB still enables to address the most significant vRES energy drops, while during flood periods, the hydropower plants are constrained to operate as run-of-the-river assets with no storage capacity.

These findings are generalized in Table~\ref{tab:profit_inflow_periods}, which reports the average ex-post market profit achieved by the proposed CADMMB under varying inflow conditions. The simulation hours are grouped into three inflow regimes: energetic, low-flow, and flood periods, with corresponding average profits reported. The highest profits occur during energetic periods, when ample storage enables full utilization of the hydropower flexibility. In contrast, the limited flexibility during low-flow and flood periods significantly reduces the VPP profitability.

\begin{figure}[!t]
\centering
\includegraphics[scale=.35]{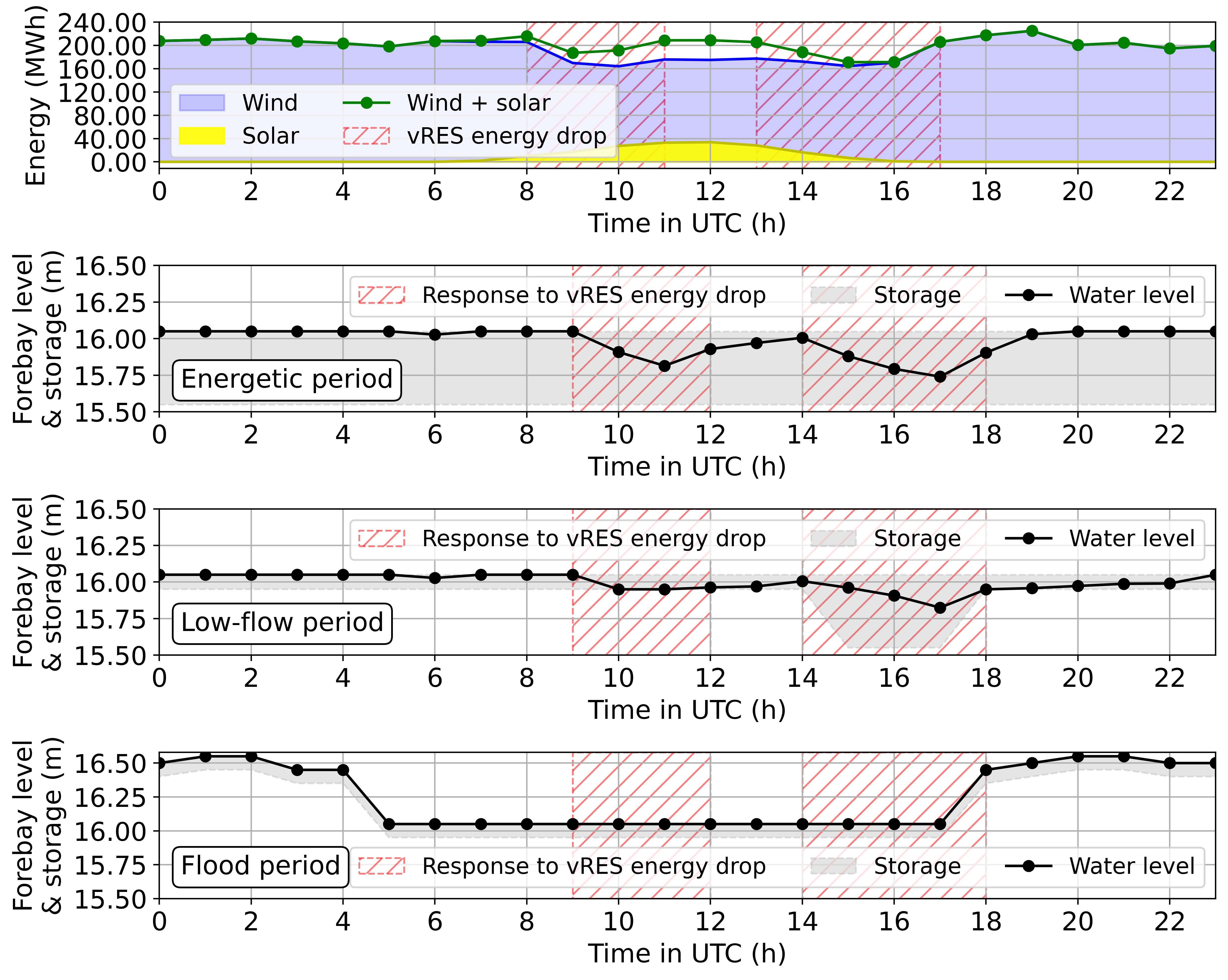}
    \caption{Example of coordination between the hydropower cascade and vRES under various inflow conditions when the proposed CADMMB is employed.}
    \label{fig:hydro_vres_example}
\end{figure}

\begin{table}[!t]
\caption{Comparison of ex-post profits observed across three distinct inflow conditions: energetic, low-flow, and flood.}
\label{tab:profit_inflow_periods}
\centering
\begin{tabular}{|c|c c c|}
\hline
\multirow{ 2}{*}{\textbf{Num. scenarios}} & \multicolumn{3}{c|}{\textbf{Profit per MWh of energy produced (\euro)}} \\
\cline{2-4}
& \textbf{Energetic} & \textbf{Low-flow} & \textbf{Flood}\\
\hline
5 & 38.81 & 33.56 & 31.16\\
10 & 41.58 & 35.99 & 33.13\\
15 & 44.88 & 38.56 & 34.79\\
20 & 48.03 & 40.05 & 35.31\\
25 & 48.63 & 40.39 & 35.32\\
\hline
\end{tabular}
\end{table}

\section{Conclusion and Future Work}
\label{sec:conclusions}
This study introduces a stochastic ADMM-based distributed algorithm for the co-optimization of vRES and CHPP in the day-ahead market. First, the original MILP problem is decomposed across spatial and scenario dimensions. A distributed strategy is then applied to derive trustworthy bounds on the global optimum, with mathematical proof provided to demonstrate their existence and validity. Unlike a conventional ADMM application, which can lead to heuristic methods for non-convex problems, the proposed CADMMB algorithm offers a performance guarantee. Testing in collaboration with the French aggregator Compagnie Nationale du Rhône shows the algorithm’s effectiveness in managing complex aggregations. The CADMMB achieves a 35\% reduction in average runtime compared to consensus ADMM and a 57\% reduction compared to traditional centralized optimization. Additionally, CADMMB consistently delivers solutions with over 99.99\% accuracy, whereas both centralized and consensus ADMM methods fail to solve the problem. Future research will aim to refine initial bound guess strategies, explore the effects of parameter tuning, extend the analysis to multi-market trading case, and integrate asynchronous communication to enhance
the robustness of the proposed distributed algorithm.

\vfill

\end{document}